\documentclass[a4paper,10pt]{amsart}

\usepackage{microtype}
\usepackage{fullpage}
\usepackage{amsmath,amsfonts,amssymb}
\usepackage[utf8x]{inputenc}
\usepackage[english]{babel}
\usepackage[T1]{fontenc}
\usepackage{lmodern}
\usepackage{textcomp}
\usepackage{graphicx}
\usepackage[all]{xy}
\usepackage{latexsym}
\usepackage{fancyhdr}
\usepackage{lscape}
\usepackage{url}
\usepackage{ifthen}
\usepackage{booktabs}
\usepackage{paralist}
\usepackage{array}
\usepackage{multirow}
\usepackage{epstopdf}
\usepackage{ifpdf}
\usepackage{caption}
\usepackage{subcaption}
\usepackage{multicol}
\usepackage{xcolor}

\usepackage[hypertexnames=false,
backref=page,
    pdftex,
    pdfpagemode=UseNone,
    breaklinks=true,
    extension=pdf,
    colorlinks=true,
    linkcolor=blue,
    citecolor=blue,
    urlcolor=blue,
]{hyperref}

\ifpdf
\fi

 \graphicspath{./Immagini/}

\newcommand{\referenza}{}

\newtheorem{prop}{Proposition}[section]
\newtheorem*{prop*}{Proposition \referenza}
\newtheorem{thm}[prop]{Theorem}
\newtheorem*{thm*}{Theorem \referenza}
\newtheorem{cor}[prop]{Corollary}
\newtheorem*{cor*}{Corollary \referenza}
\newtheorem{lemma}[prop]{Lemma}

\theoremstyle{definition}
\newtheorem{defi}[prop]{Definition}
\newtheorem{rem}[prop]{Remark}
\newtheorem{ex}[prop]{Example}

\newtheorem{question}[prop]{Question}
\newtheorem*{question*}{Question \referenza}

\newcommand{\N}{\mathbb{N}}
\newcommand{\Z}{\mathbb{Z}}

\newcommand{\R}{\mathbb{R}}
\newcommand{\C}{\mathbb{C}}

\newcommand{\sspace}{\text{\--}}
\newcommand{\ssspace}{\text{\textdblhyphen}}

\newcommand{\g}{\mathfrak{g}}

\newcommand{\heis}{\mathfrak{heis}}

\newcommand{\Sf}{\mathbb{S}}

\newcommand{\caso}[1]{\smallskip \noindent \texttt{{#1}}. }

\DeclareMathOperator{\im}{i}
\DeclareMathOperator{\imm}{im}
\DeclareMathOperator{\de}{d}

\DeclareMathOperator{\esp}{e}

\newcommand{\del}{\partial}
\newcommand{\delbar}{\overline{\del}}

\newcommand{\lcK}{lcK}
\newcommand{\gcK}{gcK}
\newcommand{\lcb}{lcb}
\newcommand{\gcb}{gcb}
\newcommand{\uG}{$1$G}
\newcommand{\kG}{$k$G}

\newcommand{\lcht}{lcht}
\newcommand{\gcht}{gcht}
\newcommand{\lcs}{lcs}
\newcommand{\gcs}{gcs}

\renewcommand{\Im}{\mathsf{Im}}
\renewcommand{\Re}{\mathsf{Re}}

\title{Locally conformal Hermitian metrics on complex non-K\"ahler manifolds}

\author{Daniele Angella}

\address[Daniele Angella]{Istituto Nazionale di Alta Matematica\\
at Departamento de Matem\'aticas\\
Universidad de Zaragoza\\
Edificio de Matem\'aticas\\
c/ Pedro Cerbuna 12, 50009\\
Zaragoza, Spain}

\curraddr{Centro di Ricerca Matematica ``Ennio de Giorgi''\\
Collegio Puteano, Scuola Normale Superiore\\
Piazza dei Cavalieri, 3\\
56100 Pisa, Italy\\
}

\email{daniele.angella@gmail.com}
\email{daniele.angella@sns.it}

\author{Luis Ugarte}

\address[Luis Ugarte]{Departamento de Matem\'aticas-I.U.M.A.\\
Universidad de Zaragoza\\
Campus Plaza San Francisco\\
50009, Zaragoza, Spain}

\email{ugarte@unizar.es}

\keywords{Complex manifold, locally conformal K\"ahler, balanced metric, locally conformal balanced, holomorphic-tamed, $\partial\overline{\partial}$-Lemma, nilmanifold, solvmanifold}
\thanks{The first author is granted with a research fellowship by Istituto Nazionale di Alta Matematica INdAM, and is supported by the Project PRIN ``Variet\`a reali e complesse: geometria, topologia e analisi armonica'', by the Project FIRB ``Geometria Differenziale e Teoria Geometrica delle Funzioni'', by SNS GR project ``Geometry of non-K\"ahler manifolds'', and by GNSAGA of INdAM.
The second author is supported by Projects MINECO (Spain) MTM2011-28326-C02-01 and MTM2014-58616-P.\\[10pt]
To appear in {\em Mediterranean Journal of Mathematics}, DOI: 10.1007/s00009-015-0586-3.
The final publication is available at \url{www.springerlink.com}
}
\subjclass[2010]{32Q99, 53C55, 53C30}

\begin{document}

\begin{abstract}
 We study complex non-K\"ahler manifolds with Hermitian metrics being locally conformal to metrics with special cohomological properties. In particular, we provide examples where the existence of locally conformal holomorphic-tamed structures implies the existence of locally conformal K\"ahler metrics, too.
\end{abstract}

\maketitle

\section*{Introduction}
\noindent A central problem in Geometry is the search of the (notion of) ``best'' metric. In K\"ahler geometry, one is led to search K\"ahler metrics with special curvature properties. In this direction, the celebrated theorem by S.-T. Yau, \cite{yau}, solving the Calabi conjecture, is one foundational example. But not every complex manifold admits a K\"ahler metric.
Therefore, in complex {\em non-K\"ahler geometry}, one has the further problem of restricting the class of Hermitian metrics to a suitable sub-class. Such sub-classes are usually characterized by {\em cohomological properties} of their associated form. For example, on a complex manifold $X$ of dimension $n$, the Hermitian metric associated to the $(1,1)$-form $\omega$ is called {\em balanced} if $\de\omega^{n-1}=0$ \cite{michelsohn}, {\em pluriclosed} if $\del\delbar\omega=0$ \cite{bismut}, {\em astheno-K\"ahler} if $\del\delbar\omega^{n-2}=0$ \cite{jost-yau}, {\em Gauduchon} if $\del\delbar\omega^{n-1}=0$ \cite{gauduchon}, and more generally {\em $k$-Gauduchon} for $k\in\{1, \ldots, n-1\}$ if $\del\delbar\omega^k\wedge\omega^{n-k-1}=0$ \cite{fu-wang-wu}.

Metrics being locally conformal to metrics with some special cohomological properties have arisen interest, too. A first reason is that, even if not every compact complex surface admits a K\"ahler metric (this depends on the parity of the first Betti number), many of them admit a metric being locally conformal to a K\"ahler metric \cite{belgun}. A second reason is that one of the equations in the Strominger system in heterotic string theory \cite{strominger} requires the existence of a metric being conformal to a balanced metric. As a third reason, P. Gauduchon proved in \cite[Th\'{e}or\`{e}me 1]{gauduchon} that every Hermitian metric on a compact complex manifold has a standard representative (called Gauduchon metric) in its conformal class.

\medskip

In this note, we study complex (non-K\"ahler) manifolds with Hermitian metrics being locally conformal to metrics with special cohomological properties. In particular, we focus on {\em locally conformal holomorphic-tamed} (also called {\em locally conformal Hermitian symplectic}) structures, providing results and examples for which the existence of these structures assures the existence of a {\em locally conformal K\"ahler}
structure. This happens for example for $6$-dimensional nilmanifolds endowed with a left-invariant complex structure, (see Theorem \ref{thm:6-nilmfd-lcht--lck},) while the Inoue surface $\mathcal{S}^{\pm}$ provides a counterexample, (see \cite[Theorem 7]{belgun} and \cite[Theorem 1.1]{apostolov-dloussky}). A related question has been formulated in \cite[Problem 1.3]{apostolov-dloussky}.

\smallskip

In Section \ref{sec:lck}, we consider locally conformal K\"ahler structures, see \cite{dragomir-ornea} and the references therein.
By \cite[Theorem 2.1, Remark (1)]{vaisman}, locally conformal K\"ahler metrics on compact complex manifolds satisfying
the {\em $\partial\overline{\partial}$-Lemma} are in fact globally conformal to a K\"ahler metric.
Hence, the deformations of the {\em holomorphically parallelizable Nakamura manifold} \cite{nakamura}
investigated in \cite{angella-kasuya-2} do not admit any locally conformal K\"ahler structure,
see Example \ref{ex:def-1-nakamura}.
They are not even in {\em class $\mathcal{C}$ of Fujiki}, see \cite[Remark 6.3]{angella-kasuya-2}: this is in support to \cite[Standard Conjecture 1.17]{popovici-annsns}, compare also \cite[Question 1.5]{popovici-inventiones}.
Finally, on a compact complex manifold with a locally conformal K\"ahler structure $\omega$, we consider the commutation relations between the naturally associated twisted differential
operators and the operators associated to the pointly linear symplectic structure, providing a sort of K\"ahler identities for compact locally conformal K\"ahler manifolds, see Proposition \ref{prop:lck-kahler-identities}, see also \cite{vanle-vanzura}.

\smallskip

In Section \ref{sec:lcb}, we consider metrics being {\em locally conformal to a balanced metric}. We prove that, as in the K\"ahler case, the property of the $\partial\overline{\partial}$-Lemma on a compact complex manifold makes any locally conformal balanced metric to be in fact globally conformal balanced. In fact, the following stronger result holds.

\renewcommand{\referenza}{\ref{thm:deldelbar-lcb--gcb}}
\begin{thm*}
 Let $X$ be a $2n$-dimensional compact manifold endowed with a complex structure $J$ such that the natural map $H^{n-1,n}_{BC}(X) \to H^{n-1,n}_{\delbar}(X)$ induced by the identity is injective. Then any locally conformal balanced
 structure is also globally conformal balanced.
\end{thm*}

Here
$H^{\bullet,\bullet}_{BC}(X):=\frac{\ker\del\cap\ker\delbar}{\imm\del\delbar}$
and $H^{\bullet,\bullet}_{\delbar}(X):=\frac{\ker\delbar}{\imm\delbar}$
denote, respectively, the {\em Bott-Chern cohomology} and the {\em Dolbeault cohomology} of $X$.
In particular, the problem of the existence of balanced metrics on compact complex manifolds satisfying the $\partial\overline{\partial}$-Lemma is reduced to the locally conformal level. However, to our knowledge there are no known examples of
compact complex manifolds satisfying the $\del\delbar$-Lemma and not admitting a locally conformal balanced metric.

Note that the map $H^{n-1,n}_{BC}(X) \to H^{n-1,n}_{\delbar}(X)$ being injective assures that $X$ satisfies the $(n-1,n)$-th weak $\del\delbar$-Lemma, as defined by J. Fu and S.-T. Yau in \cite{fu-yau}, but this weaker hypothesis does not suffice in Theorem \ref{thm:deldelbar-lcb--gcb},
see Proposition~\ref{prop:Jabel-weak-lcb}.

Once again in view of the study of special metrics on compact complex manifolds satisfying the $\partial\overline{\partial}$-Lemma, we study locally conformal balanced structures in connection with $k$-Gauduchon metrics, \cite{fu-wang-wu}, which provide a generalization of the notion of Gauduchon metrics, pluriclosed metrics, and astheno-K\"ahler metrics; see Proposition \ref{prop:deldelbar-ug-lcb--gck}.

\smallskip

In Section \ref{sec:lcht}, we address, at the locally conformal level, a question by T.-J. Li and W. Zhang \cite[page 678]{li-zhang} and by J. Streets and G. Tian \cite[Question 1.7]{streets-tian}. More precisely, the {\em ``tamed to compatible'' question} by S.~K. Donaldson \cite[Question 2]{donaldson} asks whether, on a compact almost-complex $4$-manifold, if there exists a taming symplectic structure
(that is, a symplectic structure being positive on the complex lines), then there exists also a compatible symplectic structure (that is, a taming symplectic structure being invariant with respect to the almost-complex structure). The analogous question for non-integrable almost-complex manifolds of dimension higher than $4$ has a negative answer, as follows from \cite{migliorini-tomassini, tomassini}. On the other hand, in the integrable case, no example of compact complex non-K\"ahler manifolds admitting a taming symplectic structure is yet known, \cite[page 678]{li-zhang}, \cite[Question 1.7]{streets-tian}. They are called non-K\"ahler {\em holomorphic-tamed} manifolds. In \cite[Theorem 3.3]{angella-tomassini-1}, it is proven that no $6$-dimensional nilmanifold endowed with a left-invariant complex structure admits a holomorphic-tamed structure, except for the torus. Such result has been generalized to the higher-dimensional case in \cite[Theorem 1.3]{enrietti-fino-vezzoni} (see also \cite{fino-kasuya}.)

Analogously, in the locally conformal setting, one can ask the following.

\renewcommand{\referenza}{\ref{conj:nilmfd-lcht--lck}}
\begin{question*}
For which compact complex manifolds, the existence of locally conformal holomorphic-tamed structures is equivalent to the existence of locally conformal K\"ahler structures?
\end{question*}

As shown in Theorem \ref{thm:4solvmfds-lcht}, and also \cite[Theorem 7]{belgun} and \cite[Theorem 1.1]{apostolov-dloussky}, the answer to Question \ref{conj:nilmfd-lcht--lck} is not always positive. Indeed, the Inoue surface $S^+_{n;p,q,r;t}=\left.\Gamma\middle\backslash{\mathrm{Sol}^\prime}^4_1\right.$ with $t\in\C\setminus\R$, (which corresponds to $\mathcal{S}^\pm$ with $q=-1$ in the notations of \cite{hasegawa-jsg}, see Table \ref{table:eq-struttura-4-solvmfds},) does not admit any locally conformal K\"ahler structure by \cite[Theorem 7]{belgun}. On the other hand, it admits locally conformal holomorphic-tamed structures. More precisely, during the preparation of this paper, a big progress on Question \ref{conj:nilmfd-lcht--lck} has been announced by V. Apostolov and G. Dloussky in the case of compact complex surfaces: in fact, they proved in \cite[Theorem 1.1]{apostolov-dloussky} that any compact complex surface with odd first Betti number admits a locally conformal holomorphic-tamed structure. In particular, Theorem \ref{thm:4solvmfds-lcht} can now be seen as a consequence of their result and of \cite[Theorem 7]{belgun}.

The problem of studying {\em locally conformal symplectic} structures which are not locally conformal K\"ahler is investigated also by G. Bazzoni and J.~C. Marrero, who provided in \cite{bazzoni-marrero} an example of a $4$-dimensional nilmanifold being locally conformal symplectic
and admitting no locally conformal K\"ahler metrics.

A class of manifolds for which Question \ref{conj:nilmfd-lcht--lck} has a positive answer is provided by $6$-dimensional nilmanifolds.

\renewcommand{\referenza}{\ref{thm:6-nilmfd-lcht--lck}}
\begin{thm*}
 Let $X$ be a $6$-dimensional nilmanifold endowed with a left-invariant complex structure. If $X$ admits a
 locally conformal holomorphic-tamed
 structure,
 then it admits also a locally conformal K\"ahler
 structure. In particular, either it is diffeomorphic to a torus,
 or to a compact quotient of $H(5) \times \mathbb{R}$, where $H(5)$
 is the five-dimensional Heisenberg group.
\end{thm*}

The proof is based on general results and on explicit computations. More precisely,
one can view nilmanifolds endowed with left-invariant complex structures and with locally conformal holomorphic tamed structures as mapping tori over contact nilmanifolds, see Theorem \ref{thm:nilmanifold-lcht--mapping-torus-contact},
see also \cite{banyaga,bazzoni-marrero-2}. This result is
related to \cite[Theorem 1, Theorem 2]{li}, where it is proven that compact co-symplectic manifolds are symplectic mapping tori, and compact co-K\"ahler manifolds are K\"ahler mapping tori. We further provide an obstruction to the differentiable structure underlying a $2$-step nilmanifold endowed with a left-invariant complex structure and with a locally conformal holomorphic-tamed structure in Theorem \ref{thm:nilmfd-2step-lcht--heisenberg}.

As a further class of examples in view of Question \ref{conj:nilmfd-lcht--lck}, we study compact complex surfaces diffeomorphic to solvmanifolds, as classified by K. Hasegawa in \cite{hasegawa-jsg}.

\renewcommand{\referenza}{\ref{thm:4solvmfds-lcht} (see also \cite[Theorem 7]{belgun} and \cite[Theorem 1.1]{apostolov-dloussky})}
\begin{thm*}
 Let $X$ be a compact complex surface diffeomorphic to a solvmanifold.
 Then $X$ admits locally conformal holomorphic-tamed structures.
 Except in the case of Inoue surface of type $\mathcal{S}^\pm$ with $q\neq0$, then $X$ admits also locally conformal K\"ahler structures.
\end{thm*}

Finally, we consider the $6$-dimensional
solvmanifolds endowed with an invariant
complex structure with holomorphically trivial canonical bundle studied in \cite{fino-otal-ugarte}, see also \cite{otal-phd}.
In a sense, they provide a first generalization of linear complex structures on nilmanifolds (see also \cite[Theorem 3.1]{cavalcanti-gualtieri}).

\renewcommand{\referenza}{\ref{cor:solmfd-otal-lcht-lck-final}}
\begin{cor*}
Let $X$ be a $6$-dimensional solvmanifold endowed with an invariant
complex structure with holomorphically trivial canonical bundle.
Then, $X$ admits a linear
locally conformal holomorphic-tamed
structure if and only if it admits a linear
locally conformal K\"ahler
structure.
\end{cor*}

More precisely, there are only three complex structures in the classification given in \cite{fino-otal-ugarte} admitting
locally conformal holomorphic-tamed
structures, see Theorem~\ref{thm:solmfd-otal-lcht-lck}.

\bigskip

\noindent{\sl Acknowledgments.}
The first author is greatly indebted to Adriano Tomassini for his constant support and encouragement. This work was written during the first author's stay at Departamento de Matem\'aticas of the Universidad de Zaragoza supported by a grant by INdAM. He would like to warmly thank all the people in Zaragoza for their kindly hospitality. Many thanks also to Antonio Otal, Giovanni Bazzoni, Vestislav Apostolov and Georges Dloussky.
We also thank the anonymous referee for useful comments that helped us to improve
the final version of the paper.

\section{Locally conformal K\"ahler structures}\label{sec:lck}

\noindent Let $X$ be a $2n$-dimensional manifold endowed with a complex structure $J$. We recall that a {\em locally conformal K\"ahler structure} (shortly, {\em \lcK}) on $X$ is given by a positive real $(1,1)$-form $\omega \in \wedge^{1,1}X \cap \wedge^2X$ such that there exists $\vartheta\in\wedge^1X$ with $\de\omega=\vartheta\wedge\omega$ and $\de\vartheta=0$, see, e.g., \cite{dragomir-ornea} and the references therein. The form $\vartheta$ is called the {\em Lee form} associated to $\omega$. If $\vartheta$ is $\de$-exact, then $\omega$ is called {\em globally conformal K\"ahler} (shortly, {\em \gcK}).
In fact, denote by $g:=\omega(\sspace, J\ssspace)$ the $J$-Hermitian metric associated to $J$ and $\omega$. We note that, if $\vartheta=\de f$ for some $f\in\mathcal{C}^\infty(X;\R)$, then the metric $\exp(-f)g$ is a K\"ahler metric in the conformal class of $g$. In particular, by the Poincar\'e lemma, every point has a neighbourhood $U$ such that $\omega\lfloor_U$ is conformal to a K\"ahler metric on $U$.

\begin{rem}
 Let $X$ be a $2n$-dimensional manifold endowed with a complex structure $J$ and with a non-degenerate $2$-form $\Omega\in\wedge^2X$. Consider the operator $L:=\Omega\wedge\sspace\colon \wedge^{\bullet}X \to \wedge^{\bullet+2}X$. Recall that, for any $k\in\Z$, the operator $L^k\colon \wedge^{n-k}X\to\wedge^{n+k}X$ is an isomorphism, \cite[Corollary 2.7]{yan}. Recall also that, for any $k\leq n-s$, the operator $L^k \colon \wedge^{s}X \to \wedge^{s+2k}X$ is injective, \cite[Corollary 2.8]{yan}. It follows that:
 \begin{itemize}
  \item if $2n=4$, then there exists always a (possibly, non $\de$-closed) form $\vartheta\in\wedge^{1}X$ such that $\de\Omega=\vartheta\wedge\Omega$;
  \item if $2n>4$, then the existence of $\vartheta\in\wedge^{1}X$ such that $\de\Omega=\vartheta\wedge\Omega$ implies also that $\de\vartheta=0$.
 \end{itemize}
\end{rem}

\medskip

Define the operator
$$ \de_\vartheta \;:=\; \de - \vartheta \wedge \sspace \colon \wedge^\bullet X \to \wedge^{\bullet+1}X \;. $$
With these notations, the condition $\de\omega=\vartheta\wedge\omega$ can be written as $\de_\vartheta\omega=0$. Note that, since $\de\vartheta=0$, then $\de_\vartheta^2=0$, and hence one can define the cohomology
$$ H_{\vartheta}^\bullet(X) \;:=\; \frac{\ker\de_\vartheta}{\imm\de_\vartheta} \;. $$
Note that, in fact, up to a gauge transform, $H_{\vartheta}^\bullet(X)$ does not depend on $\vartheta\in\wedge^1X \cap \ker\de$ but just on $[\vartheta]\in H^1_{dR}(X;\R)$. In fact, for $f\in\mathcal{C}^\infty(X;\R)$, one has
$$ \de_{\vartheta+\de f} \;=\; \exp(f) \cdot \de_\vartheta \left( \exp(-f) \cdot \sspace \right) \;. $$

\subsection{Locally conformal K\"ahler structures and \texorpdfstring{$\partial\overline{\partial}$}{partialoverlinepartial}-Lemma}

We recall the following theorem by I. Vaisman.

\begin{thm}[{\cite[Theorem 2.1, Remark (1)]{vaisman}}]
 Consider a compact complex manifold satisfying the $\del\delbar$-Lemma. Then any \lcK\ structure is also \gcK.
\end{thm}

The argument in \cite[Theorem 2.1]{vaisman} uses that the Bott-Chern cohomology class of $\de J\vartheta \in \wedge^{1,1}X$ vanishes, where $\vartheta$ is the Lee form of the \lcK\ structure and $J$ denotes the complex structure of $X$. This holds in particular if the natural map $H^{1,1}_{BC}(X) \to H^2_{dR}(X;\C)$ induced by the identity is injective. In \S\ref{subsec:lcb-deldelbar}, we will prove the following result with a weaker hypothesis than satisfying the $\del\delbar$-Lemma.

\renewcommand{\referenza}{\ref{prop:deldelbar-lck--gck}}
\begin{prop*}
 Let $X$ be a compact complex manifold of complex dimension $n$ such that the natural map $H^{n-1,n}_{BC}(X) \to H^{n-1,n}_{\delbar}(X)$ induced by the identity is injective. Then any \lcK\ structure is also \gcK.
\end{prop*}

In the following example, we consider the deformations in class {\itshape (1)} of the holomorphically-parallelizable Nakamura manifold investigated in \cite{angella-kasuya-2}: they do not admit any \lcK\ structure. In fact, we note that they are compact complex non-K\"ahler manifolds satisfying the $\del\delbar$-Lemma, except for the central fibre. Furthermore, they are not in class $\mathcal{C}$ of Fujiki, they admit a left-invariant balanced metric, they admit no pluriclosed metric, they admit no left-invariant \uG\ metric (see \S\ref{subsec:kG} for definitions).

\begin{ex}\label{ex:def-1-nakamura}
 Consider the holomorphically-parallelizable Nakamura manifold, \cite[\S2]{nakamura}, see also \cite[\S3]{debartolomeis-tomassini}, namely, $(X,J_0)$, where $X:=\left. \Gamma \middle\backslash G \right.$ is the quotient of the solvable group
 $$ G \;:=\; \C \ltimes_\phi \C^2 \qquad \text{ with } \qquad \phi(z) \;:=\; \left(\begin{array}{cc}\esp^z&0\\0&\esp^{-z}\end{array}\right) $$
 by a lattice $\Gamma$ in $G$, and $J_0$ is the natural complex structure induced by the quotient.

 By considering a set $\left\{ z^1, z^2, z^3 \right\}$ of local holomorphic coordinates on $G$, a $G$-left-invariant co-frame for $T^{1,0}_{J_0}X$ is given by
 $$
 \left\{ \begin{array}{l}
  \phi^1_0 \;:=\; \de z^1 \\[5pt]
  \phi^2_0 \;:=\; \esp^{-z^1}\, \de z^2 \\[5pt]
  \phi^3_0 \;:=\; \esp^{ z^1}\, \de z^3
 \end{array} \right. \;,
 $$
 with structure equations
 $$
 \left\{ \begin{array}{l}
  \de\phi^1_0 \;=\; 0 \\[5pt]
  \de\phi^2_0 \;=\; -\phi^1_0 \wedge \phi^2_0 \\[5pt]
  \de\phi^3_0 \;=\;  \phi^1_0 \wedge \phi^3_0
 \end{array} \right. \;.
 $$

 As in \cite[\S4]{angella-kasuya-2}, consider the curve $\left\{J_t\right\}_{t\in\Delta(0,\varepsilon)}$, for $\varepsilon>0$, of complex structures on $X$ given by deforming the complex structure $J_0$ of the holomorphically-parallelizable Nakamura manifold in the direction $t\, \frac{\del}{\del z^1} \otimes \de z^1 \in H^{0,1}_{J_0}(X;T^{1,0}_{J_0}X)$. Such deformations are denoted as case {\itshape (1)} in \cite{angella-kasuya-2}, where it is proven that, for $t\neq 0$, the compact complex manifold $(X,J_t)$ satisfies the $\del\delbar$-Lemma, \cite[Proposition 4.1]{angella-kasuya-2}.
 For $t\in\Delta(0,\varepsilon)$, the complex structure $J_t$ is associated to the $G$-left-invariant co-frame $\left\{\phi^1_t, \phi^2_t, \phi^3_t\right\}$ of $T^{1,0}_{J_t}X$ with structure equations
 $$
 \left\{ \begin{array}{l}
  \de\phi^1_t \;=\; 0 \\[5pt]
  \de\phi^2_t \;=\; -\phi^1_t \wedge \phi^2_t + t\, \phi^2_t \wedge \bar\phi^1_t \\[5pt]
  \de\phi^3_t \;=\;  \phi^1_t \wedge \phi^3_t - t\, \phi^3_t \wedge \bar\phi^1_t
 \end{array} \right. \;,
 $$
 see \cite[Table 3]{angella-kasuya-2}.

 Note that, for $t\in\Delta(0,\varepsilon)$, the complex structure $J_t$ does not admit any K\"ahler metric; more precisely, $X$ does not admit any K\"ahler structure, \cite[Theorem 5.1]{debartolomeis-tomassini}. We claim that $(X,J_t)$ is not in class $\mathcal{C}$ of Fujiki, \cite{fujiki}, see also \cite[Remark 6.3]{angella-kasuya-2}, even if it satisfies the $\del\delbar$-Lemma for $t\neq0$. (See also \cite[Theorem 9]{arapura}, \cite[Theorem 1.1]{bharali-biswas-mj}, or \cite[Theorem 3.3]{angella-kasuya-4})

 Indeed, consider the $G$-left-invariant real $(1,1)$-form
 \begin{eqnarray*}
  \omega_t &:=& \im\, \left( A\, \phi^1_t \wedge \bar\phi^1_t + B\, \phi^2_t \wedge \bar\phi^2_t + C\, \phi^3_t \wedge \bar\phi^3_t \right) \\[5pt]
  && + \left( D\, \phi^1_t \wedge \bar\phi^2_t - \bar D\, \phi^2_t \wedge \bar\phi^1_t \right)  + \left( E\, \phi^1_t \wedge \bar\phi^3_t - \bar E\, \phi^3_t \wedge \bar\phi^1_t \right) + \left( F\, \phi^2_t \wedge \bar\phi^3_t - \bar F\, \phi^3_t \wedge \bar\phi^2_t \right) \;,
 \end{eqnarray*}
 with $A,B,C\in\R$ and $D,E,F\in\C$. The form $\omega_t$ is positive (i.e., it is the $(1,1)$-form associated to a $J_t$-Hermitian metric) if and only if, \cite[page 189]{ugarte},
 $$ \left\{\begin{array}{l}
            A > 0 \\[5pt]
            B > 0 \\[5pt]
            C > 0 \\[5pt]
            AB > |D|^2 \\[5pt]
            AC > |E|^2 \\[5pt]
            BC > |F|^2 \\[5pt]
            ABC + 2\, \Re(\im \bar D E \bar F) > C\,|D|^2 + A\,|F|^2 + B\,|E|^2
    \end{array}\right. \;.
 $$

 A straightforward computation gives
 $$ \del_t\delbar_t \omega_t \;=\; -\im\,(1+t)\,(1+\bar t)\, \left(B\, \phi^{12\bar1\bar2}_t + C\, \phi^{13\bar1\bar3}_t \right) + (-1+t)\,(-1+\bar t)\, \left( F\, \phi^{12\bar1\bar3}_t - \bar F\, \phi^{13\bar1\bar2}_t \right) \;, $$
 where we have shortened, e.g., $\phi^{12\bar1\bar3}_t \;:=\; \phi^{1}_{t}\wedge\phi^{2}_{t}\wedge\bar\phi^{1}_{t}\wedge\bar\phi^{3}_{t}$.

 Take, e.g., $F=0$ and $A=B=C=1$. For such values, $\im\del_t\delbar_t\omega_t\geq 0$, therefore, by \cite[Theorem 2.3]{chiose-pams}, it follows that $(X,J_t)$ is not in class $\mathcal{C}$ of Fujiki.

 As a different argument to prove the non-K\"ahlerianity of $J_t$, we note that, by the F.~A. Belgun symmetrization trick, \cite[Theorem 7]{belgun}, if $(X,J_t)$ admits a K\"ahler metric, then it admits also a $G$-left-invariant K\"ahler metric. And hence, in particular, it admits a $G$-left-invariant pluriclosed metric, as well as a $G$-left-invariant \uG\ metric. On the other side, for the generic $\omega_t$ as above, we compute
 $$ \del_t\delbar_t\omega_t \wedge \omega_t \;=\; 2\, \left( (1+t)\, (1+\bar t)\, BC + (-1+t)\, (-1+\bar t)\, |F|^2 \right) \, \phi^{123\bar1\bar2\bar3}_t \;. $$
 Now note that, since $0 < \frac{|F|^2}{BC} < 1$, then
 \begin{eqnarray*}
  (1+t)\, (1+\bar t) + (-1+t)\, (-1+\bar t)\, \frac{|F|^2}{BC} &=& (1+|t|^2)\, (1+\frac{|F|^2}{BC}) + 2\, \Re t\, (1-\frac{|F|^2}{BC}) \\[5pt]
   &\geq& (1+|t|^2)\, (1+\frac{|F|^2}{BC}) - 2\, |t|\, (1-\frac{|F|^2}{BC}) \\[5pt]
   &=&    (1-|t|)^2 + (1+|t|)^2\, \frac{|F|^2}{BC} \\[5pt]
   &>&    0 \;,
 \end{eqnarray*}
 providing that there is no $G$-left-invariant $1$-Gauduchon metric for $t \in \Delta(0,\varepsilon)$.

 Note also that the condition $\del_t\delbar_t\omega_t=0$ for the generic $\omega_t$ implies $B=C=F=0$. In particular, $(X,J_t)$ does not admit any $G$-left-invariant pluriclosed metric, and hence, again by the F.~A. Belgun trick, \cite[Theorem 7]{belgun}, it admits no pluriclosed metric.

 Finally, aside, note also that the $J_t$-Hermitian metric with associated $(1,1)$-form $\Omega_t := \im \left( \phi^1_t\wedge\bar\phi^1_t + \phi^2_t\wedge\bar\phi^2_t + \phi^3_t\wedge\bar\phi^3_t \right)$ is a $G$-left-invariant balanced metric on $(X,J_t)$, for $t\in\Delta(0,\varepsilon)$.

 Since $(X,J_t)$, for $t\in\Delta(0,\varepsilon)\setminus\{0\}$, satisfies the $\del\delbar$-Lemma but admits no K\"ahler structure, by \cite[Theorem 2.1]{vaisman}, it follows that it admits no \lcK\ structure.
\end{ex}

\subsection{K\"ahler identities for locally conformal K\"ahler structures}
Let $X$ be a $2n$-dimensional manifold endowed with a complex structure $J$ and a \lcK\ structure $\omega$. Denote by $\vartheta\in\wedge^1X$ the Lee form of $\omega$.

Consider the operator
$$ L \;:=\; \omega\wedge\sspace\colon \wedge^{\bullet}X\to\wedge^{\bullet+2}X \;. $$
Since $\omega$ is non-degenerate, one can define the operator
$$ \Lambda \;:=\; -\iota_{\omega^{-1}} \colon \wedge^{\bullet}X\to\wedge^{\bullet-2}X \;. $$
Consider also the operator
$$ H \;:=\; \sum_{k\in\Z} (n-k)\,\pi_{\wedge^{k}X} \colon \wedge^{\bullet}X\to\wedge^{\bullet}X \;, $$
where $\pi_{\wedge^{k}X}\colon \wedge^{\bullet}X\to \wedge^kX$ denotes the natural projection. One has that $\left\langle L,\Lambda, H\right\rangle$ is a $\mathfrak{sl}(2;\C)$-representation of $\wedge^\bullet X \otimes_\R\C$, see \cite[Corollary 1.6]{yan}. In particular, one defines the space of {\em primitive forms}, namely, $P^kX:=\ker\Lambda\lfloor_{\wedge^kX}=\ker L^{n-k+1}\lfloor_{\wedge^kX}$, \cite[Corollary 2.6]{yan}, and there holds the {\em Lefschetz decomposition}
$$ \wedge^\bullet X \;=\; \bigoplus_{k\in\Z} L^{k}P^{\bullet-2k}X \;, $$
see \cite[Corollary 2.6]{yan}.

\medskip

We recall the following results.

\begin{lemma}[{\cite[Corollary 2.8, Corollary 2.7]{yan}}]
 Let $X$ be a $2n$-dimensional manifold endowed with a non-degenerate $2$-form $\Omega\in\wedge^2X$. Consider the operator $L:=\Omega\wedge\sspace \colon \wedge^{\bullet}X \to \wedge^{\bullet+2}X$. The operator $L^k$ is injective on $\wedge^sX$ for $k \leq n-s$. The operator $L^k\colon \wedge^{n-k}X\to\wedge^{n+k}X$ is an isomorphism for any $k\in\Z$.
\end{lemma}

\begin{lemma}[{Weyl identity, \cite[Proposition 1.2.31]{huybrechts}}]
 Let $X$ be a $2n$-dimensional manifold endowed with a non-degenerate $2$-form $\Omega\in\wedge^2X$. Consider an almost-complex structure $J$ on $X$ such that $g:=\Omega(\sspace, J\ssspace)$ is a $J$-Hermitian metric on $X$, see \cite[Corollary 12.7]{cannasdasilva}. Consider the operator $L:=\Omega\wedge\sspace \colon \wedge^{\bullet}X\to\wedge^{\bullet+2}X$, and the Hodge-$*$-operator $*\colon \wedge^{\bullet}X\to\wedge^{2n-\bullet}X$ associated to $g$. Then, for any $j\in\Z$ and $k\in\Z$, the {\em Weyl identity} holds:
 $$ \left. * L^j \right\lfloor_{P^kX} \;=\; \left. (-1)^{\frac{k(k+1)}{2}}\,\frac{j!}{(n-k-j)!}\, L^{n-k-j}J \right\lfloor_{P^kX} \;. $$
\end{lemma}

\begin{lemma}[{\cite[Lemma 2.3]{yan}}]
 Let $X$ be a $2n$-dimensional manifold endowed with a non-degenerate $2$-form $\Omega\in\wedge^2X$. Consider the operators $L:=\Omega\wedge\sspace \colon \wedge^{\bullet}X \to \wedge^{\bullet+2}X$ and $\Lambda:=-\iota_{\omega^{-1}} \colon \wedge^{\bullet}X \to \wedge^{\bullet-2}X$. Then, for any $j\in\Z$,
 $$ \left[ L^j, \Lambda \right] \;=\; j\, (k-n+j-1)\, L^{j-1} \;. $$
\end{lemma}

\medskip

We recall that, for $\eta\in\wedge^1X$ such that $\de\eta=0$, the differential operator $\de_\eta \colon \wedge^{\bullet}X\to\wedge^{\bullet+1}X$ is defined as $\de_\eta \;:=\; \de - \eta\wedge\sspace$. If $J$ denotes an almost-complex structure, then
$$ \de_\eta^c \;:=\; J^{-1}\de_\eta J \;. $$

In \cite{vanle-vanzura}, the following commutation result, concerning $L$, holds, more in general, for a {\em locally conformal symplectic} (shortly, {\em \lcs}) structure on a manifold $X$, namely, a non-degenerate real $2$-form $\Omega\in\wedge^2X$ such that $\de\Omega=\vartheta\wedge\Omega$ with $\de\vartheta=0$.

\begin{lemma}[{\cite[Equation (2.5)]{vanle-vanzura}}]
Let $X$ be a manifold endowed with a \lcs\ structure $\Omega$. Consider the operator $L:=\Omega\wedge\sspace \colon \wedge^{\bullet}X\to\wedge^{\bullet+2}X$. Then, for any $k\in\Z$ and $\ell\in\Z$,
$$ \de_{(\ell+k)\vartheta} L^k \;=\; L^k \de_{\ell\vartheta} \;. $$
\end{lemma}

We prove now the following commutation result concerning $\Lambda$: it could be compared with the K\"ahler identities in the K\"ahler case.

\begin{prop}\label{prop:lck-kahler-identities}
 Let $X$ be a $2n$-dimensional manifold endowed with a complex structure $J$ and a \lcK\ structure $\omega\in\wedge^2X$ with Lee form $\vartheta\in\wedge^1X\cap\ker\de$.
 For any $j\in\Z$ and $k\in\Z$ and $\ell\in\Z$, it holds
 $$ \left.\left( \Lambda \de_{\ell\vartheta} - \de_{(\ell-1)\vartheta}\Lambda \right)\right\lfloor_{L^jP^kX} \;=\; \left. - \left(\de^c_{(n+\ell-k-2j)\vartheta}\right)^* \right\lfloor_{L^jP^kX} $$
\end{prop}

\begin{proof}
 Consider $\alpha^{(k)}\in P^kX$, and consider the Lefschetz decomposition of $\de_{(\ell-j)\vartheta} \alpha^{(k)} \in \wedge^{k+1}X$:
 $$ \de_{(\ell-j)\vartheta} \alpha^{(k)} \;=\; \sum_{h\in\Z} L^h \beta^{(k+1-2h)} $$
 where $\beta^{(k+1-2h)} \in P^{k+1-2h}X$.
 By computing
 $$ 0 \;=\; \de_{(\ell-j+n-k+1)\vartheta} L^{n-k+1} \alpha^{(k)} \;=\; L^{n-k+1} \de_{(\ell-j)\vartheta} \alpha^{(k)} \;=\; \sum_{h\in\Z} L^{n-k+1+h} \beta^{(k+1-2h)} $$
 one gets that $L^{n-k+1+h} \beta^{(k+1-2h)}=0$ for any $h\in\Z$. In particular, since $L^\ell\lfloor_{\wedge^{s}X}$ is injective for $\ell \leq n-s$, see \cite[Corollary 2.8]{yan}, one gets that $\beta^{(k+1-2h)}=0$ for any $h\geq2$. Therefore we reduce to
 $$ \de_{(\ell-j)\vartheta} \alpha^{(k)} \;=\; \beta^{(k+1)} + L \beta^{(k-1)} \;. $$

 We compute:
 \begin{eqnarray*}
  \Lambda \de_{\ell\vartheta} (L^j\alpha^{(k)}) &=& \Lambda L^j \de_{(\ell-j)\vartheta} \alpha^{(k)} \;=\; \Lambda L^j \beta^{(k+1)} + \Lambda L^{j+1} \beta^{(k-1)} \\[5pt]
  &=& L^{j} \Lambda \beta^{(k+1)} - j(k+1-n+j-1) L^{j-1}\beta^{(k+1)} + L^{j+1} \Lambda \beta^{(k-1)} - (j+1)(k-1-n+j) L^{j}\beta^{(k-1)} \\[5pt]
  &=& - j(k-n+j) L^{j-1}\beta^{(k+1)} - (j+1)(k-n+j-1) L^{j}\beta^{(k-1)}
 \end{eqnarray*}
 and
 \begin{eqnarray*}
  \de_{(\ell-1)\vartheta} \Lambda (L^j\alpha^{(k)}) &=& \de_{(\ell-1)\vartheta} L^j \Lambda \alpha^{(k)} - j (k-n+j-1) \de_{(\ell-1)\vartheta} L^{j-1}\alpha^{(k)} \\[5pt]
  &=& - j (k-n+j-1) L^{j-1} \de_{(\ell-j)\vartheta} \alpha^{(k)} \\[5pt]
  &=& - j (k-n+j-1) L^{j-1} \beta^{(k+1)} - j (k-n+j-1) L^{j} \beta^{(k-1)} \;.
 \end{eqnarray*}

 Hence
 $$ \left( \Lambda \de_{\ell\vartheta} - \de_{(\ell-1)\vartheta} \Lambda \right) (L^j\alpha^{(k)}) \;=\; - j L^{j-1} \beta^{(k+1)} + (n-k-j+1) L^j \beta^{(k-1)} \;. $$

 On the other hand, we compute
 \begin{eqnarray*}
  - \left(\de^c_{(n+\ell-k-2j)\vartheta}\right)^* (L^j\alpha^{(k)}) &=& * J^{-1} \de_{(n+\ell-k-2j)\vartheta} J * L^j \alpha^{{(k)}} \\[5pt]
   &=& (-1)^{\frac{k(k+1)}{2}} \frac{j!}{(n-k-j)!} * J^{-1} \de_{(n+\ell-k-2j)\vartheta} J L^{n-k-j} J \alpha^{(k)} \\[5pt]
   &=& (-1)^{\frac{k(k+1)}{2}+k} \frac{j!}{(n-k-j)!} J^{-1} * \de_{(n+\ell-k-2j)\vartheta} L^{n-k-j} \alpha^{(k)} \\[5pt]
   &=& (-1)^{\frac{k(k+1)}{2}+k} \frac{j!}{(n-k-j)!} J^{-1} * L^{n-k-j} \de_{(\ell-j)\vartheta} \alpha^{(k)} \\[5pt]
   &=& (-1)^{\frac{k(k+1)}{2}+k} \frac{j!}{(n-k-j)!} J^{-1} * L^{n-k-j} \beta^{(k+1)} \\[5pt]
   && + (-1)^{\frac{k(k+1)}{2}+k} \frac{j!}{(n-k-j)!} J^{-1} * L^{n-k-j+1} \beta^{(k-1)} \\[5pt]
   &=& (-1)^{\frac{k(k+1)}{2}+k} \frac{j!}{(n-k-j)!} J^{-1} \left((-1)^{\frac{(k+2)(k+1)}{2}} \frac{(n-k-j)!}{(j-1)!} L^{j-1}J\beta^{(k+1)} \right) \\[5pt]
   && + (-1)^{\frac{k(k+1)}{2}+k} \frac{j!}{(n-k-j)!} J^{-1} \left( (-1)^{\frac{(k-1)k}{2}} \frac{(n-k-j+1)!}{j!} L^j J \beta^{(k-1)} \right) \\[5pt]
   &=& (-1)^{\frac{k(k+1)}{2}+k+\frac{(k+2)(k+1)}{2}} j J^{-1} J L^{j-1} \beta^{(k+1)} \\[5pt]
   && + (-1)^{\frac{k(k+1)}{2}+k+\frac{(k-1)k}{2}} (n-k-j+1) J^{-1} J L^j \beta^{(k-1)} \\[5pt]
   &=& - j L^{j-1}\beta^{(k+1)} + (n-k-j+1) L^j \beta^{(k-1)} \;.
 \end{eqnarray*}

Comparing the two expressions, we get the statement.
\end{proof}

\section{Locally conformal balanced structures}\label{sec:lcb}

\noindent Let $X$ be a $2n$-dimensional manifold endowed with a complex structure $J$. A {\em locally conformal balanced} (shortly, {\em \lcb}) structure on $X$ is the datum of a positive real $(1,1)$-form $\omega\in\wedge^{1,1}X\cap\wedge^2X$ (that is, $g:=\omega(\sspace,J\ssspace)$ is a $J$-Hermitian metric on $X$) such that there exists $\vartheta \in \wedge^1X$ satisfying
$$ \de \omega^{n-1} \;=\; \vartheta \wedge \omega^{n-1} \qquad \text{ and } \qquad \de \vartheta \;=\; 0 \;, $$
see, e.g., \cite{medori-tomassini-ugarte}.
Note that a \lcb\ structure is just the datum of a Hermitian metric being locally conformal to a balanced metric.
The form $\vartheta$ is called the associated {\em (balanced) Lee form}.
A {\em globally conformally balanced} (shortly, {\em \gcb}) structure is a \lcb\ structure such that the associated Lee form is $\de$-exact.

Obviously, in dimension $2n=4$, the notions of \lcK\ structure and of \lcb\ structure coincide.
On the other side, if $\omega$ is a \lcK\ structure on $X$, then note that $\de \omega^{n-1} = (n-1)\, \vartheta \wedge \omega^{n-1}$. Therefore we have the following obvious result.

\begin{prop}\label{prop:lck--lcb}
 Let $X$ be a complex manifold. A \lcK\ structure is also \lcb. Furthermore, a \lcK\ structure is \gcK\ if and only if it is \gcb.
\end{prop}

\begin{rem}\label{rem:lcb-conformal}
 We note that the property of being \lcb\ is a conformal property of Hermitian metrics on complex manifolds.
\end{rem}

\begin{rem}
 Let $X$ be $2n$-dimensional manifold endowed with a complex structure. Suppose $n\geq 4$. Consider a Hermitian metric $g$ with associated $(1,1)$-form $\omega$. Fix $s\in\{2, \ldots, n-2\}$. Suppose that there exists $\vartheta_{(s)} \in \wedge^1X$ such that
 $$ \de \omega^{s} \;=\; \vartheta_{(s)} \wedge \omega^{s} \qquad \text{ with } \qquad \de \vartheta_{(s)} \;=\; 0 \;. $$
 Then we can write this equality as
 $$ \left( s\, \de \omega - \vartheta_{(s)}\wedge \omega \right) \wedge \omega^{s-1} \;=\; 0 \;. $$
 Since the map $L_{\omega}\colon \wedge^qX \to \wedge^{q+2}X$ is injective for any $q\leq n-1$, and $L_{\omega^{s-1}}\colon \wedge^{n-s+1}X \to \wedge^{n+s-1}X$ is an isomorphism, see, \cite[Corollary 2.8, Corollary 2.7]{yan}, we have that $L_{\omega^{s-1}}\colon \wedge^{q}X \to \wedge^{q+2s-2}X$ is injective for every $q\leq n-s+1$. Since $s\, \de \omega - \vartheta_{(s)}\wedge \omega$ is a $3$-form in the kernel of $L_{\omega^{s-1}}$, and since $3\leq n-s+1$, then necessarily $\de \omega = \left(\frac{1}{s}\, \vartheta_{(s)}\right) \wedge \omega$, i.e., the structure $\omega$ is lcK.
\end{rem}

\begin{rem}
 Let $X$ be $2n$-dimensional compact manifold endowed with a complex structure $J$. Consider a \lcb\ structure $\omega$ on $X$, and let $\vartheta\in\wedge^1X$ be the $\de$-closed $1$-form such that $\de\omega^{n-1}=\vartheta\wedge\omega^{n-1}$. Consider the $J$-Hermitian metric $g:=\omega(\sspace,J\ssspace)$ associated to $\omega$, and denote its Levi Civita connection by $\nabla^{LC}$.
 Note that, if $\nabla^{LC}\vartheta=0$, then in particular $\vartheta$ is harmonic with respect to $\omega$. In particular, $g$ is also a Gauduchon metric. More precisely, by \cite[Th\'{e}or\`{e}me 1]{gauduchon}, in the conformal class of any \lcb\ structure, there is at most one \lcb\ structure with parallel Lee form, being the Gauduchon metric.
\end{rem}

\subsection{Locally conformal balanced structures and \texorpdfstring{$\partial\overline{\partial}$}{partiloverlinepartial}-Lemma}\label{subsec:lcb-deldelbar}

As in \cite[Theorem 2.1, Remark (1)]{vaisman}, we have the following result.

\begin{thm}\label{thm:deldelbar-lcb--gcb}
 Let $X$ be a $2n$-dimensional compact manifold endowed with a complex structure $J$ such that the natural map $H^{n-1,n}_{BC}(X) \to H^{n-1,n}_{\delbar}(X)$ induced by the identity is injective. Then any \lcb\ structure is also \gcb.
\end{thm}

\begin{proof}
 Consider a \lcb\ structure $\tilde{\omega}$ on $X$, and denote by $\tilde{g}:=\tilde{\omega}(\sspace,J\ssspace)$ its associated $J$-Hermitian metric. By \cite[Th\'{e}or\`{e}me 1]{gauduchon}, in the conformal class of $\tilde{g}$, there exists a metric $g:=\exp(f)\tilde{g}$ being Gauduchon, where $f\in\mathcal{C}^\infty(X;\R)$ is a smooth real function on $X$. That is to say, the $(1,1)$-form $\omega=\exp(f)\tilde{\omega}$ associated to $g$ satisfies $\del\delbar\omega^{n-1}=0$.

 Consider the form $\delbar\omega^{n-1} \in \wedge^{n-1,n}X$. Being $\del$-closed and $\delbar$-closed, it defines a class in $H^{n-1,n}_{BC}(X)$. The form $\delbar\omega^{n-1}$ being $\delbar$-exact, its class in the Bott-Chern cohomology maps to the zero class in the Dolbeault cohomology group $H^{n-1,n}_{\delbar}(X)$ under the natural map induced by the identity. By the hypothesis, it follows that $\left[\delbar\omega^{n-1}\right] \in H^{n-1,n}_{BC}(X)$ is the zero class in the Bott-Chern cohomology, that is, there exists $\eta\in\wedge^{n-2,n-1}X$ such that
 $$ \delbar\omega^{n-1} \;=\; -\del\delbar\eta \;. $$

 The structure $\omega$ is still \lcb, see Remark \ref{rem:lcb-conformal}. That is, there exists a $\de$-closed $1$-form $\vartheta\in\wedge^1 X$ such that $\de \omega^{n-1} = \vartheta \wedge \omega^{n-1}$. Consider the splitting $\vartheta = \vartheta^{1,0} + \overline{\vartheta^{1,0}}$, where $\vartheta^{1,0} \in \wedge^{1,0}X$. Note that, $\vartheta$ being $\de$-closed, then $\del\vartheta^{1,0}=0$. In particular, we have
 $$ \delbar \omega^{n-1} \;=\; \overline{\vartheta^{1,0}} \wedge \omega^{n-1} \;. $$

 By comparing the two expressions for $\delbar\omega^{n-1}$ and by wedging with $\vartheta^{1,0}$, we get
 $$ \vartheta^{1,0} \wedge \overline{\vartheta^{1,0}} \wedge \omega^{n-1} \;=\; \de \left( \delbar\eta \wedge \vartheta^{1,0} \right) \;, $$
 since $\del\vartheta^{1,0}=0$ and $\delbar\eta\wedge\vartheta^{1,0}\in\wedge^{n-1,n}X$.

 We claim that
 $$ \im\, \vartheta^{1,0} \wedge \overline{\vartheta^{1,0}} \wedge \omega^{n-1} \;=\; \varphi\, \omega^n $$
 for a smooth real non-negative function $\varphi$ on $X$, and that $\varphi$ is zero at every point if and only if $\vartheta^{1,0}=0$.
 Indeed, since $\omega$ is non-degenerate, there exists a unique smooth real function $\varphi$ on $X$ such that $\im\, \vartheta^{1,0} \wedge \overline{\vartheta^{1,0}} \wedge \omega^{n-1} = \varphi\, \omega^n$, and it suffices to prove that $\varphi$ is pointly non-negative and that $\varphi(x)=0$ at a point $x\in X$ if and only if $\vartheta^{1,0}\lfloor_{x}=0$.
 Fix a point $x\in X$ and consider a basis $\left\{ \tau^1, \ldots, \tau^n \right\}$ of the $\C$-vector space $\left(T^{1,0}_xX\right)^*$ such that $\omega\lfloor_{x} = \im\, \sum_{j=1}^{n} A_j\, \tau^j \wedge \bar\tau^j$ with $A_j>0$. Let $\alpha_1,\ldots,\alpha_n\in\C$ be such that $\left. \vartheta^{1,0} \right\lfloor_{x} = \sum_{k=1}^{n} \alpha_k\, \tau^k$. Then we compute
 $$ \left. \left( \im\, \vartheta^{1,0} \wedge \overline{\vartheta^{1,0}} \wedge \omega^{n-1} \right) \right\lfloor_{x}
 \;=\; \im^n\, n!\cdot \left(\prod_{j=1}^{n}A_j\right)\cdot\left(\sum_{k=1}^{n} \frac{\left|\alpha_k\right|^2}{A_k} \tau^1 \wedge \overline{\tau^1} \wedge \cdots \wedge \tau^n \wedge \overline{\tau^n} \right)
 \;=\; \sum_{k=1}^{n} \frac{\left|\alpha_k\right|^2}{A_k} \, \left. \left( \omega^n \right) \right\lfloor_x \;. $$
 This proves the claim.

 Now, by the Stokes theorem, we get
 $$ 0 \;=\; \im\, \int_X \de \left( \delbar\eta \wedge \vartheta^{1,0} \right) \;=\; \int_X \im\, \vartheta^{1,0} \wedge \overline{\vartheta^{1,0}} \wedge \omega^{n-1} \;=\; \int_X \varphi\, \omega^n \;. $$
 Hence $\varphi$ is zero at every point, from which it follows that $\vartheta^{1,0}=0$. This means that $g$ is actually balanced, and hence $\tilde{\omega}$ is a \gcb\ structure.
\end{proof}

\begin{rem}
 Note that the property that the natural map $H^{n-1,n}_{BC}(X) \to H^{n-1,n}_{\delbar}(X)$ induced by the identity is injective in Theorem \ref{thm:deldelbar-lcb--gcb} is weaker than the property of satisfying $\del\delbar$-Lemma. For example, consider the completely-solvable Nakamura manifold endowed with the complex structure in case {\itshape (ii)} as in \cite[Example 1]{kasuya-mathz} (where it is denoted as case {\itshape (B)}), and \cite[Example 2.17]{angella-kasuya-1}, see \cite[Example 1]{nakamura}. Its Dolbeault cohomology is computed at \cite[page 445]{kasuya-mathz}, and its Bott-Chern cohomology is computed in \cite[Table 5]{angella-kasuya-1}. It follows that it does not satisfy the $\del\delbar$-Lemma, but the natural map $H^{2,3}_{BC}(X) \to H^{2,3}_{\delbar}(X)$ is an isomorphism.
\end{rem}

As a corollary, we get the following, to be compared with \cite[Theorem 2.1, Remark (1)]{vaisman}.

\begin{prop}\label{prop:deldelbar-lck--gck}
 Let $X$ be a compact complex manifold of complex dimension $n$ such that the natural map $H^{n-1,n}_{BC}(X) \to H^{n-1,n}_{\delbar}(X)$ induced by the identity is injective. Then any \lcK\ structure is also \gcK.
\end{prop}

\begin{proof}
 Let $\omega$ be a \lcK\ structure on the complex manifold $X$. By Proposition \ref{prop:lck--lcb}, $\omega$ is also \lcb. By Theorem \ref{thm:deldelbar-lcb--gcb}, $\omega$ is \gcb. Again by Proposition \ref{prop:lck--lcb}, it follows that $\omega$ is actually \gcK.
\end{proof}

\medskip

We note that, up to our knowledge, the known examples of manifolds satisfying the $\del\delbar$-Lemma are actually balanced: K\"ahler manifolds are clearly balanced; manifolds in class $\mathcal{C}$ of Fujiki, \cite{fujiki}, and Mo\v\i\v shezon manifolds, \cite{moishezon}, are balanced by \cite[Corollary 5.7]{alessandrini-bassanelli}; twistor spaces \cite{penrose, atiyah-hitchin-singer} are balanced by \cite[Proposition 11]{gauduchon-annsns}, and their small deformations are balanced by \cite[Corollary 9]{fu-yau}; small deformations of a complex manifold that admits a balanced metric and satisfies the $\del\delbar$-Lemma still admit balanced metrics by \cite[Theorem 5.13]{wu}, see also \cite[Theorem 6]{fu-yau}. (See also \cite{popovici-annsns} for a recent survey.)
One can ask the following.

\begin{question}\label{conj:deldelbar--lcb}
 Does every compact complex manifold satisfying the $\del\delbar$-Lemma admit a \lcb\ structure?
\end{question}

\subsection{Locally conformal balanced structures and \texorpdfstring{$(n-1,n)$}{(n-1,n)}-th weak-\texorpdfstring{$\partial\overline{\partial}$}{partiloverlinepartial}-Lemma}

In \cite[Definition 5]{fu-yau}, a compact complex manifold $X$ of complex dimension $n$ is said to {\em satisfy the $(n-1,n)$-th weak $\del\delbar$-Lemma}
if for each real form $\alpha$ of type $(n-1,n-1)$ such that $\delbar\alpha$ is $\del$-exact there exists a $(n-2,n-1)$-form $\beta$ such that $\delbar\alpha=\im\,\del\delbar\beta$.

\medskip

The hypothesis that the natural map $H^{n-1,n}_{BC}(X) \to H^{n-1,n}_{\delbar}(X)$ induced by the identity is injective in Theorem \ref{thm:deldelbar-lcb--gcb} implies in particular the $(n-1,n)$-th weak $\del\delbar$-Lemma. Next we see that the converse is not true. Notice that it is proven in \cite[Corollary 3.5]{ugarte-villacampa} that any left-invariant Abelian complex structure on a nilmanifold satisfies the $(n-1,n)$-th weak $\del\delbar$-lemma.

\begin{prop}\label{prop:weak-lemma-for-abelian}
Let $X$ be a $2n$-dimensional nilmanifold endowed with a left-invariant Abelian complex structure.
Then $X$ satisfies the $(n-1,n)$-th weak $\del\delbar$-Lemma, but the natural map $H^{n-1,n}_{BC}(X) \to H^{n-1,n}_{\delbar}(X)$ induced by the identity is never injective.
\end{prop}

\begin{proof}
The validity of the $(n-1,n)$-th weak $\del\delbar$-Lemma is proven in \cite[Corollary 3.5]{ugarte-villacampa}.

For nilmanifolds endowed with left-invariant complex structures, the inclusion of left-invariant forms induces isomorphisms in Dolbeault, Bott-Chern, and Aeppli cohomologies, \cite[Remark 4]{console-fino}, \cite[Theorem 3.8]{angella-1}. Therefore, the injectivity of the map $H^{n-1,n}_{BC}(X) \to H^{n-1,n}_{\delbar}(X)$ can be checked at the Lie algebra level, and it is equivalent to
$$
\delbar\left(\wedge^{n-1,n-1}\mathfrak{g}^*\right) \cap \ker \del \;\subseteq\; \del\delbar\left(\wedge^{n-2,n-1}\mathfrak{g}^*\right) \;.
$$
Since $J$ is Abelian, we have $\del\left(\wedge^{n-1,n}\mathfrak{g}^*\right)=\{0\}$ and $\del\left(\wedge^{n-2,n}\mathfrak{g}^*\right)=\{0\}$, and thus also $\delbar\left(\wedge^{n-1,n-1}\mathfrak{g}^*\right) \subseteq \ker \del$ and
$\del\delbar\left(\wedge^{n-2,n-1}\mathfrak{g}^*\right)=\{0\}$.
This reduces the injectivity of the map $H^{n-1,n}_{BC}(X) \to H^{n-1,n}_{\delbar}(X)$
to the following condition:
$$ \delbar\left(\wedge^{n-1,n-1}\mathfrak{g}^*\right) \;=\; \{0\} \;. $$
But, if $\mathfrak{g}$ is not an Abelian Lie algebra, this cannot happen. Indeed, we can choose a basis $\left\{\omega^j\right\}_{j\in\{1,\ldots,n\}}$ of $\wedge^{1,0}\mathfrak{g}^*$ such that
$$ \delbar \omega^1 \;=\; \cdots \;=\; \delbar \omega^{r-1} \;=\; 0 \;, \qquad \delbar \omega^r \;=\; \omega^{1}\wedge\bar\omega^{1} \;, \qquad \delbar \omega^{r+j} \;\in\; \wedge^{2}\left\langle \omega^2, \bar\omega^2, \ldots, \omega^n, \bar\omega^n\right\rangle \text{ for } j>0\;, $$
for some $r\in\{2,\ldots,n\}$.
Therefore $\delbar\left(\omega^2\wedge\bar\omega^2\wedge\cdots\wedge\omega^n\wedge\bar\omega^n\right)\neq 0$.
\end{proof}

In the following result we prove that Theorem \ref{thm:deldelbar-lcb--gcb} cannot be extended to compact complex manifolds satisfying the $(n-1,n)$-th weak $\del\delbar$-Lemma.

\begin{prop}\label{prop:Jabel-weak-lcb}
 There exist compact complex manifolds $X$ of complex dimension $n$ satisfying the $(n-1,n)$-th weak $\del\delbar$-Lemma and having \lcb\ metrics, but not admitting any  balanced metric.
\end{prop}

\begin{proof}
Because of Proposition \ref{prop:weak-lemma-for-abelian}, we consider $X$ a nilmanifold of dimension $6$
with underlying Lie algebra $\mathfrak{h}_8$, $\mathfrak{h}_9$, or $\mathfrak{h}_{15}$, and endowed with an Abelian complex structure.
It is proven in \cite[Proposition 1.1]{medori-tomassini-ugarte} that $X$ has left-invariant \lcb\ structures; however $X$ does not admit any balanced metric by \cite[Theorem 26]{ugarte}.
\end{proof}

\subsection{Locally conformal balanced and \texorpdfstring{$k$}{k}-Gauduchon structures}\label{subsec:kG}

Let $X$ be a $2n$-dimensional manifold endowed with a complex structure. We recall that, for $k\in\{1,\ldots,n-1\}$, a {\em $k$-Gauduchon} (shortly, {\em \kG}) structure on $X$ is the datum of the $(1,1)$-form $\omega$ associated to a Hermitian metric on $X$ such that
$$ \del\delbar\omega^k\wedge\omega^{n-k-1} \;=\; 0 \;, $$
see \cite[Definition 1]{fu-wang-wu}. Note that $(n-1)$-Gauduchon structures are associated to Gauduchon metrics, \cite{gauduchon}, and that, for $n=2$, Gauduchon metrics are $1$-Gauduchon, as well as pluriclosed, \cite{bismut}. See also \cite{ivanov-papadopulos, fino-ugarte} for further generalizations and results.

\medskip

By \cite[Proposition 1.4]{fino-parton-salamon}, see also \cite[Remark 1]{alexandrov-ivanov}, on a compact complex manifold, Hermitian metrics being both pluriclosed and balanced are in fact K\"ahler. By \cite[Proposition 2.4]{fino-ugarte}, on a compact complex manifold, Hermitian metrics being both \uG\ and balanced are in fact K\"ahler. As a sort of generalization, also in view of Question \ref{conj:deldelbar--lcb}, we have the following result.

\begin{prop}\label{prop:deldelbar-ug-lcb--gck}
 Let $X$ be a compact complex manifold of complex dimension $n$ such that the natural map $H^{n-1,n}_{BC}(X) \to H^{n-1,n}_{\delbar}(X)$ induced by the identity is injective. Then, for $k\in\{1, \ldots, n-2\}$, any \kG\ \lcb\ structure is also \gcK.
\end{prop}

\begin{proof}
 By \cite[Corollary 4]{fu-wang-wu}, for any $k\in\{1,\ldots,n-1\}$, to any Hermitian metric $g$ with associated $(1,1)$-form $\omega$, it is associated a unique constant $\gamma_k(g)\in\R$. The sign of $\gamma_k(g)$ is invariant in the conformal class of $g$ by \cite[Proposition 11]{fu-wang-wu}. Furthermore, on the one hand, for $k\in\{1,\ldots,n-1\}$, by \cite[Corollary 4]{fu-wang-wu}, if $\omega$ is a \kG\ structure, then $\gamma_k(g)=0$. On the other hand, for $k\in\{1, \ldots, n-2\}$, by \cite[Lemma 3.7]{ivanov-papadopulos}, if $g$ is balanced non-K\"ahler, then $\gamma_k(g)>0$. In particular, if $\omega$ is \gcb\ non-\gcK, then $\gamma_k(g)>0$.

 Suppose that $\omega$ is a \kG\ \lcb\ non-\gcK\ structure on $X$ with associated Hermitian metric $g$. By the hypothesis and by Theorem \ref{thm:deldelbar-lcb--gcb}, $\omega$ is \kG\ \gcb\ non-\gcK. Since $\omega$ is \kG, one has $\gamma_k(g)=0$. Since $\omega$ is \gcb\ non-\gcK, one has $\gamma_k(g)>0$. This is absurd.
\end{proof}

\subsection{Locally conformal balanced structures on solvmanifolds}

We consider the existence of locally conformally balanced structures on solvmanifolds, namely, compact quotients of connected simply-connected solvable Lie groups by co-compact discrete subgroups.

As in \cite[\S1]{sawai} for \lcK\ structures, we have that, on completely-solvable solvmanifolds endowed with left-invariant complex structures, it suffices to study the existence of left-invariant \lcb\ structures.

\begin{prop}
 Let $X = \left. \Gamma \middle\backslash G \right.$ be a completely-solvable solvmanifold endowed with a $G$-left-invariant complex structure. If there exists a \lcb\ structure, then there exists also a $G$-left-invariant \lcb\ structure.
\end{prop}

\begin{proof}
 Denote the real dimension of $X$ by $2n$, and the complex structure on $X$ by $J$. Denote the Lie algebra associated to $G$ by $\mathfrak{g}$, and identify the $G$-left invariant forms on $X$ with the space $\wedge^{\bullet}\mathfrak{g}^*$. Let $\omega$ be a \lcb\ structure on $X$ with associated Hermitian metric $g:=\omega(\sspace,J\ssspace)$. By definition, there exists a $\de$-closed $1$-form $\vartheta$ such that $\de \omega^{n-1} = \vartheta \wedge \omega^{n-1}$.

 We can assume that $\vartheta$ is $G$-left-invariant. Indeed, by A. Hattori's theorem \cite[Corollary 4.2]{hattori}, $\vartheta$ is cohomologous to a $G$-left-invariant $1$-form: let $\vartheta_{\text{inv}}\in\wedge^1\mathfrak{g}^*$ and $f\in\mathcal{C}^\infty(X;\R)$ be such that $\vartheta = \vartheta_{\text{inv}} + \de f$. Consider the metric $\hat{g} := \exp\left(-\frac{f}{n-1}\right) g$. Then the associated $(1,1)$-form $\hat{\omega} = \exp\left(-\frac{f}{n-1}\right) \omega$ to $\hat{g}$ satisfies $\de\hat{\omega}^{n-1} = \vartheta_{\text{inv}} \wedge \hat{\omega}^{n-1}$ for the $G$-left-invariant $\de$-closed $1$-form $\vartheta_{\text{inv}}$.

 Hence, assume that the $\de$-closed $1$-form $\vartheta$ is $G$-left-invariant. Consider the F.~A. Belgun symmetrization map, \cite[Theorem 7]{belgun},
 $$ \mu\colon \wedge^\bullet X \to \wedge^\bullet\mathfrak{g}^* \;, \qquad \mu(\alpha)\;:=\;\int_X \alpha\lfloor_x \, \eta(x) \;, $$
 where $\eta$ is a $G$-bi-invariant volume form on $G$ such that $\int_X\eta=1$, \cite[Lemma 6.2]{milnor}. If $\eta \in \wedge^\bullet\mathfrak{g}^*$ and $\alpha \in \wedge^\bullet X$, then $\mu(\eta\wedge\alpha) = \eta \wedge \mu(\alpha)$, see, e.g., \cite[Lemma 2.5]{angella-kasuya-1}. Consider the $G$-left-invariant metric $g_{\text{inv}}$ whose associated $(1,1)$-form $\omega_{\text{inv}}\in\wedge^2\mathfrak{g}^* \cap \wedge^{1,1}X$ satisfies $\omega_{\text{inv}}^{n-1} := \mu(\omega^{n-1})$: it exists by the M.~L. Michelsohn trick, \cite[pages 279--280]{michelsohn}. It satisfies
 $$ \de\omega_{\text{inv}}^{n-1} \;=\; \de \left( \mu \left( \omega^{n-1} \right) \right) \;=\; \mu \left( \de\omega^{n-1} \right) \;=\; \mu\left( \vartheta \wedge \omega^{n-1} \right) \;=\; \vartheta \wedge \mu \left( \omega^{n-1} \right) \;=\; \vartheta \wedge \omega_{\text{inv}}^{n-1} \;, $$
 since $\de\circ\mu=\mu\circ\de$ by \cite[Theorem 7]{belgun}. Hence $\omega_{\text{inv}}$ is a $G$-left-invariant \lcb\ structure on $X$.
\end{proof}

\begin{ex}[A \uG\ \lcb\ non-balanced manifold with $\Delta^{2n-1}=0$]\label{ex:1g-lcb-nolck}
 Consider a $6$-dimensional nilmanifold $X$ with associated Lie algebra $\mathfrak{h}_8:=(0,0,0,0,0,12)$, in the notation of \cite{salamon}. Up to equivalence, it admits only one left-invariant complex structure, \cite[Corollary 15]{ugarte}. By \cite[Theorem 1.2, Theorem 3.2]{fino-parton-salamon}, any left-invariant Hermitian metric on $X$ is pluriclosed, and hence in particular \uG. By \cite[Proposition 1.1]{medori-tomassini-ugarte}, $X$ admits a left-invariant \lcb\ structure. On the other hand, by \cite[Theorem A]{benson-gordon}, or \cite[Theorem 1, Corollary]{hasegawa}, $X$ admits no K\"ahler structure.

Notice that $\Delta^5=0$ (see \cite[\S6]{latorre-ugarte-villacampa} or \cite[Table 2]{angella-franzini-rossi}), and that it does not admit any balanced metric.
(Here, $\Delta^k$ is the non-negative degree $\Delta^k:=\sum_{p+q=k}\dim_\C \left(H^{p,q}_{BC}(X)+\dim_\C H^{p,q}_{A}(X)\right)-b_k\in\N$, see \cite[Theorem A]{angella-tomassini-3}.) This shows in particular that, on compact complex manifolds of complex dimension $n$, the condition $\Delta^{2n-1}=0$ and the existence of \lcb\ structures do not imply the existence of balanced metrics.
\end{ex}

\section{Locally conformal (Hermitian) symplectic structures}\label{sec:lcht}

\noindent Let $X$ be a $2n$-dimensional manifold. We recall that a {\em locally conformal symplectic} (shortly, {\em \lcs}) structure on $X$ is given by a non-degenerate $2$-form $\Omega$ such that $\de\Omega=\vartheta\wedge\Omega$ for some $\vartheta\in\wedge^1X$ with $\de\vartheta=0$, \cite[\S1]{vaisman-ijmms}. If $\vartheta$ is $\de$-exact, then $\Omega$ is called {\em globally conformal symplectic} (shortly, {\em \gcs}): in fact, if $\vartheta=\de f$ for $f\in\mathcal{C}^\infty(X;\R)$, then the structure $\exp(-f)\Omega$ is actually symplectic.

\subsection{Locally conformal symplectic structures as mapping tori}
In this section, we review when \lcs\ manifolds can be seen as mapping tori over contact manifolds. First of all, we recall the definition of mapping torus.

\begin{defi}[{\cite[page 527]{li}}]
 Let $S$ be a compact manifold (possibly endowed with a further structure). Consider a diffeomorphism $\varphi$ on $S$ (preserving the possible further structure). The {\em mapping torus} $S_\varphi$ on $S$ is the manifold
 $$ S_\varphi \;:=\; \left. \left( S \times [0,1] \right) \middle\slash \left\{ (x,0) \sim (\varphi(x),1) \right\} \right. \;. $$
\end{defi}

We recall also that a {\em contact structure} on a $(2n-1)$-dimensional manifold $X$ is the datum of a $1$-form $\alpha\in\wedge^1X$ such that $\de\alpha^{n-1}\wedge\alpha \neq 0$ everywhere, see, e.g., \cite{blair}.

The following result proves that mapping tori over contact manifolds are endowed with a \lcs\ structure.

\begin{prop}[{\cite[Example 2]{banyaga}}]
 Let $S$ be a $(2n-1)$-dimensional compact manifold endowed with a contact structure $\alpha\in\wedge^1X$. Consider a diffeomorphism $\varphi$ of $S$ such that $\varphi^*\alpha=\alpha$.
 Then the mapping torus $S_\varphi$ on $S$ has a \lcs\ structure.
\end{prop}

\begin{proof}
 For the sake of completeness, we recall here the construction.

 Consider the projection map $\pi_S \colon S \times [0,1] \to S$. Define the form $\beta:=\pi_S^*\alpha \in \wedge^1\left(S\times[0,1]\right)$. In fact, since $\varphi^*\alpha=\alpha$, we can consider $\beta\in\wedge^1S_\varphi$ such that $\de\beta^{n-1}\wedge\beta\neq 0$ everywhere.

 One has that $S_\varphi$ is the total space of a fibre bundle $S \hookrightarrow S_\varphi \stackrel{\pi}{\to} \Sf^1$. Consider a coordinate $t$ on $\Sf^1$, and define $\vartheta := \pi^*(\de t) \in \wedge^1S_\varphi$. Note that $\de\vartheta=0$.

 Define $\Omega := \de_\vartheta\beta = \de\beta - \vartheta\wedge\beta \in \wedge^2S_\varphi$. Note that
 $$ \de_\vartheta\Omega \;=\; 0 \;. $$
 Note also that, for $k\in\N$, it holds $\Omega^k = \phi^k + k\, \phi^{k-1} \wedge \beta \wedge \vartheta$, where $\phi:=\de\beta$. In particular,
 $$ \Omega^n \;=\; n\, \phi^{n-1} \wedge \beta \wedge \vartheta \neq 0 $$
 everywhere, since $\phi^{n-1} \wedge \beta$ and $\vartheta$ are independent.
 Hence $\Omega$ is a \lcs\ structure on $S_\varphi$, with associated Lee form $\vartheta$.
\end{proof}

Conversely, let us assume further hypotheses to guarantee that manifolds with \lcs\ structures are mapping tori over contact manifolds. The argument is inspired by \cite{li}, where H. Li proves that co-symplectic manifolds are symplectic mapping tori, \cite[Theorem 1]{li}, and that co-K\"ahler manifolds are K\"ahler mapping tori, \cite[Theorem 2]{li}.

\begin{prop}[{\cite[Theorem 2]{banyaga}}]\label{prop:lcs--mapping-torus}
 Consider a $2n$-dimensional compact manifold $X$ endowed with a \lcs\ structure $\Omega\in\wedge^2X$ with everywhere non-vanishing Lee form $\vartheta\in\wedge^1X$. Suppose that $\Omega$ is $\de_\vartheta$-exact, namely, there exists $\beta\in\wedge^1X$ such that $\Omega=\de\beta-\vartheta\wedge\beta$. Then $X$ has the structure of mapping torus over a $(2n-1)$-dimensional manifold $S$ endowed with a contact structure $\alpha\in\wedge^1X$.
\end{prop}

\begin{proof}
 For the sake of completeness, we show here the construction.

 Since $\vartheta$ is a $\de$-closed $1$-form on $X$, by D. Tischler's theorem, \cite[Theorem 1]{tischler}, there exist a $(2n-1)$-dimensional manifold $S$ and a diffeomorphism $\hat\varphi$ on $S$ such that $X$ is the mapping torus $S_{\hat\varphi}$ on $S$, where $S_{\hat\varphi}$ is the total space of a fibre bundle $S \hookrightarrow S_{\hat\varphi} \stackrel{\pi}{\to} \Sf^1$ with $\vartheta=\pi^*(\de t)$, where $t$ is a coordinate on $\Sf^1$.

 For any $\tau\in[0,1]$, consider the map $\iota_\tau \colon S \ni x \mapsto (x,\tau) \in S \times [0,1]$. Consider also the quotient projection map $\rho\colon S \times [0,1] \to S_{\hat\varphi}$. Define
 $$ \tilde\alpha_\tau \;:=\; (\rho\circ\iota_\tau)^*\beta \in \wedge^1 S \;. $$

 Note that
 $$ \de\tilde\alpha_\tau^{n-1} \wedge \tilde\alpha_\tau \;\neq\; 0 $$
 everywhere.
 In particular, $\tilde\alpha_\tau$, for any $\tau\in[0,1]$, is a contact structure on $S$.
 Indeed, take a local trivialization chart $U\times K$ on $S_{\hat\varphi}$ such that $\pi\lfloor_{U\times K}$ is the projection on the $K$ factor. Denote the coordinate on $K$ by $t$. Note that $0 \neq \Omega^n\lfloor_{U \times K} = \de\beta^{n-1}\wedge\beta\wedge\de t \lfloor_{U\times K} = (\rho\circ\iota_\tau)^*\left(\de\beta^{n-1}\wedge\beta \lfloor_{U} \right) \wedge \de t \lfloor_{K}$ everywhere. Hence $\de\tilde\alpha_\tau^{n-1} \wedge \tilde\alpha_\tau = (\rho\circ\iota_\tau)^*\left(\de\beta^{n-1}\wedge\beta\right) \neq 0$ everywhere.

 We claim now that there is a diffeomorphism $\varphi$ of $S$ isotopic to $\hat\varphi$ such that $\varphi^*\tilde\alpha_0=\tilde\alpha_0$.
 Firstly, note that $\iota_1 \circ {\hat\varphi} = \iota_0$. It follows that ${\hat\varphi}^*\tilde\alpha_1=\tilde\alpha_0$. Hence ${\hat\varphi}^* \left( \de\tilde\alpha_1^{n-1} \wedge \tilde\alpha_1 \right) =  \de\tilde\alpha_0^{n-1} \wedge \tilde\alpha_0$. Since $\iota_0$ and $\iota_1$ are homotopic, then $\left[\de\tilde\alpha_0^{n-1}\wedge\tilde\alpha_0\right] = \left[\de\tilde\alpha_1^{n-1}\wedge\tilde\alpha_1\right]$. By the Gray stability theorem, see, e.g., \cite[pages 60-61]{geiges}, there exists a diffeomorphism of $S$ isotopic to the identity such that $\tilde\alpha_1 = F^*\left(\tilde\alpha_0\right)$. By taking $\varphi:=F \circ \hat\varphi$, one gets that $\varphi^* ( \tilde\alpha_0 ) = \tilde\alpha_0$, proving the claim.

 In particular, $\alpha:=\tilde\alpha_0$ is a contact structure on $S_{\hat\varphi}$. Finally, note that, since $\hat\varphi$ and $\varphi$ are isotopic, then $S_\varphi$ and $S_{\hat\varphi}$ are diffeomorphic as fibre bundles.
\end{proof}

In the next section, we will apply the previous results to the case of nilmanifolds, namely, compact quotients of connected simply-connected nilpotent Lie groups by co-compact discrete subgroups, see \ref{thm:nilmanifold-lcht--mapping-torus-contact}.

\subsection{Locally conformal holomorphic-tamed structures on nilmanifolds}

We recall that a {\em holomorphic-tamed} (also called {\em Hermitian symplectic}) structure on a manifold $X$ endowed with a complex structure $J$ is the datum of a symplectic structure $\Omega$ taming $J$, that is, such that $g_J:=\frac{1}{2}\left(\Omega(\sspace, J\ssspace)-\Omega(J\sspace,\ssspace)\right)$ is a $J$-Hermitian metric on $X$. Up to now, no example of complex holomorphic-tamed non-K\"ahler manifold is known, see \cite[page 678]{li-zhang}, \cite[Question 1.7]{streets-tian}: in a sense, this is the integrable higher-dimensional analogue of the Donaldson ``tamed to compatible question'' for compact almost-complex $4$-manifolds, \cite[Question 2]{donaldson}; (compare also \cite{migliorini-tomassini, tomassini}).

In particular, recall that non-tori nilmanifolds never admit K\"ahler structures, \cite[Theorem A]{benson-gordon}, see also \cite[Theorem 1, Corollary]{hasegawa}. In \cite[Theorem 3.3]{angella-tomassini-1} for dimension $6$, and in \cite[Theorem 1.3]{enrietti-fino-vezzoni} in full generality, it is proven that non-tori nilmanifolds endowed with left-invariant complex structures admit no holomorphic-tamed structures, too. A further generalization for completely-solvable solvmanifolds has been recently obtained by A. Fino and H. Kasuya, \cite[Theorem 3.1]{fino-kasuya}.

Therefore, we are interested in studying the locally conformal analogue of holomorphic-tamed structures on nilmanifolds.

\begin{defi}
Let $X$ be a manifold endowed with a complex structure $J$. A non-degenerate $2$-form $\Omega\in\wedge^2X$ on $X$ is called {\em locally conformal holomorphic-tamed} (shortly, {\em \lcht}; also called {\em locally conformal Hermitian symplectic}) if:
\begin{inparaenum}[\itshape (i)]
 \item $\Omega$ tames the complex structure $J$, namely, for any $x\in X$, for any $v_x\in T_xX\setminus\{0\}$, it holds $\Omega_x(v_x,Jv_x)>0$; in other words, $g_J := \frac{1}{2} \left(\Omega(\sspace,J\ssspace)-\Omega(J\sspace,\ssspace)\right)$ is a $J$-Hermitian metric on $X$;
 \item there exists $\vartheta\in\wedge^1X$ such that $\de\Omega=\vartheta\wedge\Omega$ and $\de\vartheta=0$.
\end{inparaenum}
The $1$-form $\vartheta$ is called the {\em Lee form} associated to $\Omega$. If such a $\vartheta$ is $\de$-exact, then $\Omega$ is called {\em globally conformal holomorphic-tamed} (shortly, {\em \gcht}).
\end{defi}

\begin{rem}
Note that, if $\Omega$ is $J$-invariant, that is, $\Omega\in\wedge^2X\cap\wedge^{1,1}X$, then $\Omega$ is in fact \lcK. Note also that, if $\Omega$ is \gcht, i.e., there exists $f\in\mathcal{C}^\infty(X;\R)$ such that $\vartheta=\de f$, then the symplectic structure $\tilde\Omega:=\exp(-f)\Omega$ is in fact a holomorphic-tamed structure in the same conformal class of $\Omega$.
\end{rem}

\medskip

The following lemma allows to reduce, on completely-solvable solvmanifolds with left-invariant complex structures, to considering just left-invariant \lcht\ structures.

\begin{lemma}\label{lemma:nilmfd-lcht--inv-lcht}
 Let $X = \left. \Gamma \middle\backslash G \right.$ be a solvmanifold endowed with a $G$-left-invariant complex structure. Denote by $\mathfrak{g}$ the Lie algebra associated to $G$. Assume that the inclusion $\wedge^\bullet\mathfrak{g}^* \hookrightarrow \wedge^\bullet X$ induces the isomorphism $H^1\left(\mathfrak{g}\right) \stackrel{\simeq}{\to} H^1_{dR}(X;\R)$. If $X$ admits a \lcht\ structure $\Omega$ such that $\de\Omega=\vartheta\wedge\Omega$, then it admits also a $G$-left-invariant \lcht\ structure $\hat\Omega$ such that $\de\hat\Omega=\hat\vartheta\wedge\hat\Omega$, with $[\vartheta]=[\hat\vartheta]$.
\end{lemma}

Note that the condition $H^1\left(\mathfrak{g}\right) \stackrel{\simeq}{\to} H^1_{dR}(X;\R)$ holds, for example, when $X$ is a completely-solvable solvmanifold, by A. Hattori's theorem, \cite[Corollary 4.2]{hattori}.

\begin{proof}
 By the hypothesis $H^1\left(\mathfrak{g}\right) \stackrel{\simeq}{\to} H^1_{dR}(X;\R)$ there is a $G$-left-invariant form $\hat\vartheta$ being cohomologous to $\vartheta$: let $f\in\mathcal{C}^\infty(X;\R)$ be such that $\vartheta=\hat\vartheta+\de f$. Then $\tilde\Omega:=\exp(-f)\Omega$ is a \lcht\ structure, in the conformal class of $\Omega$, satisfying $\de\tilde\Omega=\hat\vartheta\wedge\tilde\Omega$. Consider the F.~A. Belgun symmetrization map, \cite[Theorem 7]{belgun},
 $$ \mu\colon \wedge^\bullet X \to \wedge^\bullet\mathfrak{g}^* \;, \qquad \mu(\alpha)\;:=\;\int_X \alpha\lfloor_x \, \eta(x) \;, $$
 where $\eta$ is a $G$-bi-invariant volume form on $G$ such that $\int_X\eta=1$, \cite[Lemma 6.2]{milnor}. By \cite[Lemma 2.5]{angella-kasuya-1}, if $\eta \in \wedge^\bullet\mathfrak{g}^*$ and $\alpha \in \wedge^\bullet X$, then $\mu(\eta\wedge\alpha) = \eta \wedge \mu(\alpha)$. In particular, $\hat\Omega:=\mu(\tilde\Omega)$ is a $G$-left-invariant \lcht\ structure satisfying $\de\hat\Omega=\hat\vartheta\wedge\hat\Omega$.
\end{proof}

We use now the results in the previous section to get a description of nilmanifolds endowed with left-invariant complex structures and with \lcht\ structures as mapping tori over contact nilmanifolds.

\begin{thm}\label{thm:nilmanifold-lcht--mapping-torus-contact}
 Let $X = \left. \Gamma \middle\backslash G \right.$ be a nilmanifold endowed with a $G$-left-invariant complex structure and a \lcht\ structure $\Omega$. Then either $\Omega$ is a \gcht\ structure on $X$, or it induces on $X$ a structure of mapping torus over a $(2n-1)$-dimensional contact nilmanifold.
\end{thm}

\begin{proof}
 Consider the Lee form $\vartheta\in\wedge^1X$ of $\Omega$, namely, $\de\Omega=\vartheta\wedge\Omega$ with $\de\vartheta=0$. By Nomizu's theorem, \cite[Theorem 1]{nomizu}, we can apply Lemma \ref{lemma:nilmfd-lcht--inv-lcht} and hence we may assume that $\vartheta$ is $G$-left-invariant.
 Note that either $\vartheta=0$, and hence $\Omega$ is a symplectic structure on $X$, i.e., a \gcht\ structure, or $\vartheta$ is everywhere non-vanishing. Hence assume now that $\vartheta$ is everywhere non-vanishing.

 Again by Lemma \ref{lemma:nilmfd-lcht--inv-lcht}, we may suppose that also $\Omega$ is $G$-left-invariant.
 Since $\Omega\in\wedge^2\mathfrak{g}^*$ is $\de_\vartheta$-closed, by \cite[Th\'{e}or\`{e}me 1]{dixmier}, see also \cite[Corollary 2.5]{millionshchikov}, one has that $\Omega$ is $\de_\vartheta$-exact.
 By Proposition \ref{prop:lcs--mapping-torus}, $X$ has a structure of mapping torus $S_\varphi$ on a $(2n-1)$-dimensional manifold $S$ endowed with a contact structure. Note that $S$ is in fact a nilmanifold.
\end{proof}

\medskip

We recall that, by \cite[Main Theorem]{sawai}, a non-torus nilmanifold endowed with a left-invariant complex structure admits a \lcK\ structure if and only if it is biholomorphic to a quotient of $\left( H(2n-1) \times \R,\, J_+ \right)$. Here $H(2n-1)$ is the $(2n-1)$-dimensional Heisenberg group, with Lie algebra $\heis_{2n-1}=\left\langle e_1,\ldots, e_{2n-2}, e_{2n} \right\rangle$ with $\left[e_{2j-1},e_{2j}\right]=-e_{2n}$ for any $j\in\{1,\ldots,n-1\}$, and $\left[e_h,e_k\right]=0$ otherwise. Consider $\heis_{2n-1}\times\R$ where $\R=\left\langle e_{2n-1} \right\rangle$: it is endowed with a linear (integrable) complex structure $J_+$ defined as $J_+e_{2j-1}=e_{2j}$ for $j\in\{1,\ldots,n\}$ (see also \cite[Table 1]{ceballos-otal-ugarte-villacampa}, where $\heis_{2n-1}\times\R$ for $n=3$ is $\mathfrak{h}_3$).

Analogously, we prove that $2$-step nilmanifolds endowed with left-invariant complex structures admit a \lcht\ structure if and only if they are diffeomorphic to a quotient of $H(2n-1)\times\R$.

\begin{thm}\label{thm:nilmfd-2step-lcht--heisenberg}
 Let $X = \left. \Gamma \middle\backslash G \right.$ be a $2n$-dimensional nilmanifold. Suppose that the Lie algebra $\g$ associated to $X$ is $2$-step. If $X$ admits a $G$-left-invariant complex structure $J$ and a \lcht\ structure, then $\g\simeq\heis_{2n-1}$.
\end{thm}

\begin{proof}
 Suppose that the Lie algebra associated to $X$ is isomorphic to $\heis_{2n-1}\times\R$. By \cite[Main Theorem]{sawai}, the $G$-left-invariant complex structure $J_+$ on $X$ admits a \lcK\ structure, namely, $\Omega := \de e^{2n} - (-e^{2n-1}) \wedge e^{2n} = \sum_{j=1}^{n} e^{2j-1} \wedge e^{2j}$, where $\left\{e^j\right\}_j$ denotes the dual basis to $\left\{e_j\right\}_j$.

 Suppose now that $X$ is a $2$-step nilmanifold endowed with a $G$-left-invariant complex structure and with a \lcht\ structure. By Lemma \ref{lemma:nilmfd-lcht--inv-lcht}, $X$ admits a $G$-left-invariant \lcht\ structure $\Omega$ such that $\de\Omega=\vartheta\wedge\Omega$ with $\vartheta\in\wedge^1\mathfrak{g}^*$.
 Consider the $J$-Hermitian metric $g:=\frac{1}{2}\left(\Omega(\sspace,J\ssspace)-\Omega(J\sspace,\ssspace)\right)$.

 If $b_1\leq 2n-2$, and by using Nomizu's theorem \cite[Theorem 1]{nomizu}, then there exist $A,C \in [\g,\g]$ with $\R A \neq \R C$.
 Since $\g$ is $2$-step nilpotent and $\vartheta$ is $\de$-closed, it holds that
 $$ \Omega(A,\sspace) \;=\; \beta(A)\cdot\vartheta \qquad \text{ and } \qquad \Omega(C,\sspace) \;=\; \beta(C)\cdot\vartheta \;. $$
 In particular, $\beta(A)\neq0$ and $\beta(C)\neq0$. Furthermore, $B:=JA\not\in[\g,\g]$ and $D:=JC\not\in[\g,\g]$.

 Since $\R A \neq \R C$, we have that, for any $t\in\R$,
 $$ 0 \;<\; \Omega \left( t\, A + C, J(t\, A + C) \right) \;=\;
    \left( \begin{array}{cc} t & 1 \end{array} \right)
    \cdot
    \left( \begin{array}{cc} \Omega(A,JA) & \Omega(A,JC) \\ \Omega(C,JA) & \Omega(C,JC) \end{array} \right)
    \cdot
    \left( \begin{array}{c} t \\ 1 \end{array} \right) \;. $$
 On the other side, by the above formulas, one has
 \begin{eqnarray*}
  \lefteqn{\Omega(A,JA) \cdot \Omega(C,JC) - \Omega(A,JC) \cdot \Omega(C,JA)} \\[5pt]
  &=& \beta(A) \cdot \vartheta(JA) \cdot \beta(C) \cdot \vartheta(JC) - \beta(A) \cdot \vartheta(JC) \cdot \beta(C) \cdot \vartheta(JA) \;=\; 0 \;.
 \end{eqnarray*}

 By the contradiction, we get that $b_1\geq 2n-1$. In fact, $\mathfrak{g}$ being $2$-step, we have $b_1 = 2n-1$.

 By \cite[Th\'{e}or\`{e}me 1]{dixmier}, see also \cite[Corollary 2.5]{millionshchikov}, one has that $\Omega$ is $\de_\vartheta$-exact: let $\beta\in\wedge^1\mathfrak{g}^*$ be such that $\Omega = \de \beta - \vartheta \wedge \beta$. Since $\Omega^n=-n\vartheta\wedge\beta\wedge(\de\beta)^{n-1}$, we have that $\left.\de\beta\right\lfloor_{\left.\mathfrak{g}\middle\slash \R\left\langle \vartheta,\beta\right\rangle \right.}$ is non-degenerate. Then $\mathfrak{g}\simeq\heis_{2n-1}$.
\end{proof}

In view of the analogous result in \cite[Theorem 3.3]{angella-tomassini-1} for the non-conformal case, we have the following theorem.

\begin{thm}\label{thm:6-nilmfd-lcht--lck}
 Let $X$ be a $6$-dimensional nilmanifold endowed with a left-invariant complex structure. If $X$ admits a \lcht\ structure, then it admits also a \lcK\ structure. In particular, either it is diffeomorphic to a torus, namely, its Lie algebra $\mathfrak{g}$ is $\mathfrak{g}\simeq\mathfrak{h}_1 = (0,0,0,0,0,0)$, or $\mathfrak{g}\simeq\mathfrak{h}_3 = (0,0,0,0,0,12+34)$.
\end{thm}

\begin{proof}
 Let $\Omega$ be a \lcht\ structure on $X$. If the Lee form $\vartheta$ of $\Omega$ is exact, then, up to a global conformal change, $\Omega$ is a holomorphic-tamed structure on $X$. By \cite[Theorem 3.3]{angella-tomassini-1}, see also \cite[Theorem 1.3]{enrietti-fino-vezzoni}, $X$ admits in fact a K\"ahler structure, in particular, a \lcK\ structure. More precisely, $X$ is diffeomorphic to the $6$-dimensional torus, $\mathfrak{h}_1=(0,0,0,0,0,0)$, by \cite[Theorem A]{benson-gordon} or \cite[Theorem 1, Corollary]{hasegawa}.
 Therefore suppose now that $\vartheta$ is everywhere non-vanishing. By Theorem \ref{thm:nilmanifold-lcht--mapping-torus-contact}, then $X$ has a structure of mapping torus over a $5$-dimensional nilmanifold endowed with a contact structure. By \cite[\S6]{kutsak}, the only $5$-dimensional nilmanifolds admitting a contact structure are $(0,0,0,0,12+34)$, and $(0,0,0,12,14+23)$, and $(0,0,12,13,14-23)$. Moreover, the Lie algebra of $X$ has to admit a linear complex structure. It follows that the Lie algebra of $X$ has to be either $\mathfrak{h}_{3} = \mathfrak{s}_{32} = (0,0,0,0,0,12+34)$, or $\mathfrak{h}_{9} = \mathfrak{s}_{27} = (0,0,0,0,12,14+25)$, or $\mathfrak{h}_{19}^{-} = \mathfrak{s}_{16} = (0,0,0,12,23,14-35)$, in the notations of \cite{salamon, kutsak}.
 Hence, it remains to consider each case separately; in fact, they all are admissible by Theorem \ref{thm:nilmfd-2step-lcht--heisenberg}. (See also \cite[Proposition 5.6]{bazzoni-marrero-2}, where $6$-dimensional Lie algebras admitting \lcs\ structures are classified.) In \cite{ceballos-otal-ugarte-villacampa}, the left-invariant complex structures on $6$-dimensional nilmanifolds are classified up to equivalence.
 In particular, there is only one, up to equivalence, left-invariant complex structure $J$ on $\mathfrak{h}_9$. We prove in Lemma \ref{lemma:h9-J-no-lcht} than it does not admit any \lcht\ structure.
 And there are two left-invariant complex structures on $\mathfrak{h}_{3}$, which are called $J_+$ and $J_-$.
 By \cite[Main Theorem]{sawai}, $J_+$ on $\mathfrak{h}_{3}$ is the only left-invariant complex structure on $6$-dimensional nilmanifolds admitting a \lcK\ structure: for the sake of completeness, we recall the metric in Lemma \ref{lemma:h3-J+-lcht}. On the other side, we prove in Lemma \ref{lemma:h3-J--no-lcht} than $J_-$ on $\mathfrak{h}_3$ does not admit any \lcht\ structure.
 Finally, there are only two, up to equivalence, left-invariant complex structures, $J_+$ and $J_-$, on $\mathfrak{h}_{19}^{-}$.
 We prove in Lemma \ref{lemma:h19-J+--no-lcht} than they do not admit any \lcht\ structure.
 This concludes the proof.
\end{proof}

\begin{lemma}[{\cite[Main Theorem]{sawai}}]\label{lemma:h3-J+-lcht}
 Let $X$ be a $6$-dimensional nilmanifold with associated Lie algebra $\mathfrak{h}_{3} = (0,0,0,0,0,12+34)$ and endowed with the $G$-left-invariant complex structure $J_+$ defined by the $G$-left-invariant co-frame $\left\{\varphi^1,\varphi^2,\varphi^3\right\}$ of $T^{1,0}X$ with structure equations
 $$ \left\{ \begin{array}{l}
             \de\varphi^1 \;=\; 0 \\[5pt]
             \de\varphi^2 \;=\; 0 \\[5pt]
             \de\varphi^3 \;=\; \varphi^{1}\wedge\bar\varphi^{1} + \varphi^{2}\wedge\bar\varphi^{2} \;.
            \end{array}
    \right. $$
 Then $X$ admits a \lcK\ structure.
\end{lemma}

\begin{proof}
 Consider the $2$-form
 $$ \omega \;:=\; \im \varphi^1\wedge\bar\varphi^1 + \im \varphi^2\wedge\bar\varphi^2 + \im \varphi^3\wedge\bar\varphi^3 \;\in\; \wedge^{1,1}X \cap \wedge^2X \;. $$
 It clearly tames $J$ and is $J$-compatible. One computes
 \begin{eqnarray*}
  \de\omega &=& \im \varphi^1\wedge\bar\varphi^1\wedge\bar\varphi^3 + \im \varphi^2\wedge\bar\varphi^2\wedge\bar\varphi^3 - \im \varphi^1\wedge\varphi^3\wedge\bar\varphi^1 - \im \varphi^2\wedge\varphi^3\wedge\bar\varphi^2 \\[5pt]
  &=& \left( \varphi^3 + \bar\varphi^3 \right) \wedge \omega \;,
 \end{eqnarray*}
 which proves that $\omega$ is a $G$-left-invariant \lcK\ structure on $X$ with Lee form $\vartheta:= \varphi^3 + \bar\varphi^3$.
\end{proof}

\begin{lemma}\label{lemma:h3-J--no-lcht}
 Let $X$ be a $6$-dimensional nilmanifold with associated Lie algebra $\mathfrak{h}_{3} = (0,0,0,0,0,12+34)$ and endowed with the $G$-left-invariant complex structure $J_-$ defined by the $G$-left-invariant co-frame $\left\{\varphi^1,\varphi^2,\varphi^3\right\}$ of $T^{1,0}X$ with structure equations
 $$ \left\{ \begin{array}{l}
             \de\varphi^1 \;=\; 0 \\[5pt]
             \de\varphi^2 \;=\; 0 \\[5pt]
             \de\varphi^3 \;=\; \varphi^{1}\wedge\bar\varphi^{1} - \varphi^{2}\wedge\bar\varphi^{2} \;.
            \end{array}
    \right. $$
 Then $X$ does not admit any \lcht\ structure.
\end{lemma}

\begin{proof}
 By Lemma \ref{lemma:nilmfd-lcht--inv-lcht}, it suffices to prove that there is no $G$-left-invariant \lcht\ structure on $X$ satisfying $\de_\vartheta\Omega=0$ with $\vartheta$ a $G$-left-invariant $1$-form.

 By \cite[Proposition 2.3]{ugarte-villacampa}, the $(1,1)$-form associated to a $G$-left-invariant Hermitian metric on $X$ can be written as
 $$ \im r^2\, \varphi^1\wedge\bar\varphi^1 + \im s^2\, \varphi^2\wedge\bar\varphi^2 + \im t^2\, \varphi^3\wedge\bar\varphi^3 + u\, \varphi^1\wedge\bar\varphi^2 - \bar u \, \varphi^2\wedge\bar\varphi^1 $$
 with $r,s,t\in\R\setminus\{0\}$ and $u\in\C$ such that $|u|^2<s^2$.
 Hence, consider the form
 \begin{eqnarray*}
  \Omega &:=&
       A \, \varphi^1\wedge\varphi^2 + B \, \varphi^1\wedge\varphi^3 + C \, \varphi^2\wedge\varphi^3 \\[5pt]
  && + \im r^2\, \varphi^1\wedge\bar\varphi^1 + \im s^2\, \varphi^2\wedge\bar\varphi^2 + \im t^2\, \varphi^3\wedge\bar\varphi^3 + u\, \varphi^1\wedge\bar\varphi^2 - \bar u \, \varphi^2\wedge\bar\varphi^1 \\[5pt]
  && + \bar A \, \bar\varphi^1\wedge\bar\varphi^2 + \bar B \, \bar\varphi^1\wedge\bar\varphi^3 + \bar C \, \bar\varphi^2\wedge\bar\varphi^3
 \end{eqnarray*}
 with $r,s,t\in\R\setminus\{0\}$ and $u\in\C$ such that $|u|^2<s^2$ and $A,B,C\in\C$.

 Consider a $\de$-closed left-invariant $1$-form $\vartheta\in\wedge^1\mathfrak{g}^*$. It is of the form
 $$ \vartheta \;=\; \alpha\, \varphi^1 + \beta\, \varphi^2 + \gamma\, \varphi^3 + \bar\alpha\, \bar\varphi^1  + \bar\beta\, \bar\varphi^2 + \gamma\, \bar\varphi^3 $$
 for $\alpha,\beta\in\C$ and $\gamma\in\R$.

 For $(p,q)\in\Z^2$, denote by $\pi_{\wedge^{p,q}X} \colon \wedge^{\bullet,\bullet}X \to \wedge^{p,q}X$ the natural projection. We compute
 $$ \pi_{\wedge^{3,0}X} \de_\vartheta \Omega \;=\; \left( - \alpha C + \beta B - \gamma A \right) \, \varphi^{123} $$
 and
 \begin{eqnarray*}
  \pi_{\wedge^{2,1}X} \de_\vartheta \Omega &=&
       \left( \im \beta r^2 + \alpha \bar u - \bar \alpha A \right) \, \varphi^{12\bar1} + \left( - \im \alpha s^2 + \beta u - \bar \beta A + B + C \right) \, \varphi^{12\bar2} + \left( - \gamma A \right) \, \varphi^{12\bar3} \\[5pt]
  && + \left( - \im t^2 + \im \gamma r^2 - \bar\alpha B \right) \, \varphi^{13\bar1} + \left( \gamma u - \bar\beta B \right) \, \varphi^{13\bar2} + \left( - \im \alpha t^2 - \gamma B \right) \, \varphi^{13\bar3} \\[5pt]
  && + \left( - \gamma \bar u - \bar\alpha C \right) \, \varphi^{23\bar1} + \left( \im t^2 + \im \gamma s^2 - \bar\beta C \right) \, \varphi^{23\bar2} + \left( - \im \beta t^2 - \gamma C \right) \, \varphi^{23\bar3} \;.
 \end{eqnarray*}
 (As a matter of notation, we have shortened, e.g., $\varphi^{12\bar3}:=\varphi^{1}\wedge\varphi^{2}\wedge\bar\varphi^{3}$.)

 We have to find $\vartheta$ as above such that $\de_\vartheta \Omega=0$.
 Note that $\gamma\neq0$: otherwise, from the coefficient of $\varphi^{23\bar3}$, we get that also $\beta=0$, which yields that the coefficient of $\varphi^{23\bar2}$ is non-zero.
 Looking at the coefficients of $\varphi^{13\bar3}$ and of $\varphi^{23\bar3}$, we get
 $$ B \;=\; - \im \frac{\alpha t^2}{\gamma} \qquad \text{ and } \qquad C \;=\; - \im \frac{\beta t^2}{\gamma} \;. $$
 By summing the coefficients of $\varphi^{13\bar1}$ and $\varphi^{23\bar2}$ and substituting and simplifying, we get
 $$ \gamma^2 (r^2+s^2) + |\alpha|^2 t^2 + |\beta|^2 t^2 \;=\;0 \;, $$
 which is not possible. Hence, there exists no ($G$-left-invariant) \lcht\ structure on $X$.
\end{proof}

\begin{lemma}\label{lemma:h9-J-no-lcht}
 Let $X$ be a $6$-dimensional nilmanifold with associated Lie algebra $\mathfrak{h}_{9} = (0,0,0,0,12,14+25)$ and endowed with the $G$-left-invariant complex structure $J_-$ defined by the $G$-left-invariant co-frame $\left\{\varphi^1,\varphi^2,\varphi^3\right\}$ of $T^{1,0}X$ with structure equations
 $$ \left\{ \begin{array}{l}
             \de\varphi^1 \;=\; 0 \\[5pt]
             \de\varphi^2 \;=\; \varphi^{1}\wedge\bar\varphi^{1} \\[5pt]
             \de\varphi^3 \;=\; \varphi^{1}\wedge\bar\varphi^{2} + \varphi^{2}\wedge\bar\varphi^{1} \;.
            \end{array}
    \right. $$
 Then $X$ does not admit any \lcht\ structure.
\end{lemma}

\begin{proof}
 By Lemma \ref{lemma:nilmfd-lcht--inv-lcht}, it suffices to prove that there is no $G$-left-invariant \lcht\ structure on $X$ satisfying $\de_\vartheta\Omega=0$ with $\vartheta$ a $G$-left-invariant $1$-form.

 By \cite[Proposition 2.3]{ugarte-villacampa}, the $(1,1)$-form associated to a $G$-left-invariant Hermitian metric on $X$ can be written as
 $$ \im r^2\, \varphi^1\wedge\bar\varphi^1 + \im s^2\, \varphi^2\wedge\bar\varphi^2 + \im t^2\, \varphi^3\wedge\bar\varphi^3 + u\, \varphi^1\wedge\bar\varphi^2 - \bar u \, \varphi^2\wedge\bar\varphi^1 $$
 with $r,s,t\in\R\setminus\{0\}$ and $u\in\C$ such that $|u|^2<s^2$.
 Hence, consider the form
 \begin{eqnarray*}
  \Omega &:=&
       A \, \varphi^1\wedge\varphi^2 + B \, \varphi^1\wedge\varphi^3 + C \, \varphi^2\wedge\varphi^3 \\[5pt]
  && + \im r^2\, \varphi^1\wedge\bar\varphi^1 + \im s^2\, \varphi^2\wedge\bar\varphi^2 + \im t^2\, \varphi^3\wedge\bar\varphi^3 + u\, \varphi^1\wedge\bar\varphi^2 - \bar u \, \varphi^2\wedge\bar\varphi^1 \\[5pt]
  && + \bar A \, \bar\varphi^1\wedge\bar\varphi^2 + \bar B \, \bar\varphi^1\wedge\bar\varphi^3 + \bar C \, \bar\varphi^2\wedge\bar\varphi^3
 \end{eqnarray*}
 with $r,s,t\in\R\setminus\{0\}$ and $u\in\C$ such that $|u|^2<s^2$ and $A,B,C\in\C$.

 Consider a $\de$-closed left-invariant $1$-form $\vartheta\in\wedge^1\mathfrak{g}^*$. It is of the form
 $$ \vartheta \;=\; \alpha\, \varphi^1 + \beta\, \varphi^2 + \gamma\, \varphi^3 + \bar\alpha\, \bar\varphi^1  + \bar\beta\, \bar\varphi^2 + \gamma\, \bar\varphi^3 $$
 for $\alpha,\beta\in\C$ and $\gamma\in\R$.

 For $(p,q)\in\Z^2$, denote by $\pi_{\wedge^{p,q}X} \colon \wedge^{\bullet,\bullet}X \to \wedge^{p,q}X$ the natural projection. We compute
 $$ \pi_{\wedge^{3,0}X} \Omega \;=\; \left( -\alpha C + \beta B - \gamma A \right) \, \varphi^{123} $$
 and
 \begin{eqnarray*}
  \pi_{\wedge^{2,1}X} \Omega &=&
       \left( - \im s^2 - B + \alpha \bar u + \im \beta r^2 - \bar \alpha A \right) \, \varphi^{12\bar1} + \left( C - \im \alpha s^2 + \beta u - \bar\beta A \right) \, \varphi^{12\bar2} + \left( - \gamma A \right) \, \varphi^{12\bar3} \\[5pt]
  && + \left( - C + \im \gamma r^2 - \bar\alpha B \right) \, \varphi^{13\bar1} + \left( - \im t^2 + \gamma u - \bar\beta B \right) \, \varphi^{13\bar2} + \left( - \im \alpha t^2 - \gamma B \right) \, \varphi^{13\bar3} \\[5pt]
  && + \left( - \im t^2 - \gamma \bar u - \bar\alpha C \right) \, \varphi^{23\bar1} + \left( \im \gamma s^2 - \bar\beta C \right) \, \varphi^{23\bar2} + \left( - \im \beta t^2 - \gamma C \right) \, \varphi^{23\bar3} \;.
 \end{eqnarray*}
 (As a matter of notation, we have shortened, e.g., $\varphi^{12\bar3}:=\varphi^{1}\wedge\varphi^{2}\wedge\bar\varphi^{3}$.)

 Note that $\gamma\neq0$: otherwise, from the coefficient of $\varphi^{13\bar3}$ we get that also $\alpha=0$, which yields that the coefficient of $\varphi^{23\bar1}$ is non-zero.
 Looking at the coefficient of $\varphi^{23\bar3}$, we get
 $$ C \;=\; - \im \frac{\beta t^2}{\gamma} \;. $$
 By substituting in the coefficient of $\varphi^{23\bar2}$ and simplifying, we get
 $$ \gamma^2 s^2 + |\beta|^2 t^2 \;=\;0 \;, $$
 which is not possible. Hence, there exists no ($G$-left-invariant) \lcht\ structure on $X$.
\end{proof}

\begin{lemma}\label{lemma:h19-J+--no-lcht}
 Let $X$ be a $6$-dimensional nilmanifold with associated Lie algebra $\mathfrak{h}_{19}^{-} = (0,0,0,12,23,14-35)$ and endowed with the $G$-left-invariant complex structure $J_+$, respectively $J_-$, defined by the $G$-left-invariant co-frame $\left\{\varphi^1,\varphi^2,\varphi^3\right\}$ of $T^{1,0}X$ with structure equations, respectively,
 $$ \left\{ \begin{array}{l}
             \de\varphi^1 \;=\; 0 \\[5pt]
             \de\varphi^2 \;=\; \varphi^{1}\wedge\varphi^{3}+\varphi^{1}\wedge\bar\varphi^{3} \\[5pt]
             \de\varphi^3 \;=\; \pm\im\left(\varphi^{1}\wedge\bar\varphi^{2}-\varphi^{2}\wedge\bar\varphi^{1}\right) \;.
            \end{array}
    \right. $$
 Then $X$ does not admit any \lcht\ structure.
\end{lemma}

\begin{proof}
 By Lemma \ref{lemma:nilmfd-lcht--inv-lcht}, it suffices to prove that there is no $G$-left-invariant \lcht\ structure on $X$ satisfying $\de_\vartheta\Omega=0$ with $\vartheta$ a $G$-left-invariant $1$-form.
 Hence, consider the form
 \begin{eqnarray*}
  \Omega &:=&
       A \, \varphi^1\wedge\varphi^2 + B \, \varphi^1\wedge\varphi^3 + C \, \varphi^2\wedge\varphi^3 \\[5pt]
  && + \im r^2\, \varphi^1\wedge\bar\varphi^1 + \im s^2\, \varphi^2\wedge\bar\varphi^2 + \im t^2\, \varphi^3\wedge\bar\varphi^3  \\[5pt]
  && + u\, \varphi^1\wedge\bar\varphi^2 - \bar u \, \varphi^2\wedge\bar\varphi^1
  + v\, \varphi^2\wedge\bar\varphi^3 - \bar v \, \varphi^3\wedge\bar\varphi^2
  + z\, \varphi^1\wedge\bar\varphi^3 - \bar z \, \varphi^3\wedge\bar\varphi^1 \\[5pt]
  && + \bar A \, \bar\varphi^1\wedge\bar\varphi^2 + \bar B \, \bar\varphi^1\wedge\bar\varphi^3 + \bar C \, \bar\varphi^2\wedge\bar\varphi^3
 \end{eqnarray*}
 with $A,B,C,u,v,z\in\C$ and $r,s,t\in\R\setminus\{0\}$ satisfying the restrictions that ensure that the $(1,1)$-component of $\Omega$ is positive.

 Consider a $\de$-closed left-invariant $1$-form $\vartheta\in\wedge^1\mathfrak{g}^*$. It is of the form
 $$
 \vartheta \;=\; \alpha\, \varphi^1 + \gamma\, \varphi^3 + \bar\alpha\, \bar\varphi^1 + \gamma\, \bar\varphi^3
 $$
 for $\alpha\in\C$ and $\gamma\in\R$.

 For $(p,q)\in\Z^2$, denote by $\pi_{\wedge^{p,q}X} \colon \wedge^{\bullet,\bullet}X \to \wedge^{p,q}X$ the natural projection. We compute
 $$ \pi_{\wedge^{3,0}X} \de_\vartheta \Omega \;=\; -\left(\alpha C + \gamma A \right) \, \varphi^{123} $$
 and
 \begin{eqnarray*}
  \pi_{\wedge^{2,1}X} \de_\vartheta \Omega &=&
       \left( \pm\im B \mp\im z +\alpha\bar u-\bar\alpha A \right) \, \varphi^{12\bar1}
       + \im \, \left( \pm C \mp v-\alpha s^2 \right) \, \varphi^{12\bar2} \\[5pt]
  &&   - \left( \alpha v + \gamma A \right) \, \varphi^{12\bar3}
  + \left( u-\bar u+\alpha\bar z+ \im\gamma r^2-\bar\alpha B \right) \, \varphi^{13\bar1} \\[5pt]
  && + \left( \im s^2\pm t^2 + \alpha\bar v + \gamma u \right) \, \varphi^{13\bar2}
  + \left( v-C-\im\alpha t^2 + \gamma z -\gamma B \right) \, \varphi^{13\bar3} \\[5pt]
  && + \left( \im s^2\mp t^2-\gamma\bar u-\bar\alpha C \right) \, \varphi^{23\bar1}
  + \im \gamma s^2 \, \varphi^{23\bar2}
  + \gamma \, \left( v-C \right) \, \varphi^{23\bar3} \;.
 \end{eqnarray*}
 (As a matter of notation, we have shortened, e.g., $\varphi^{12\bar3}:=\varphi^{1}\wedge\varphi^{2}\wedge\bar\varphi^{3}$.)

 We have to find $\vartheta$ as above such that $\de_\vartheta \Omega=0$.
 Note that the coefficient of $\varphi^{23\bar2}$ vanishes if and only if $\gamma=0$.
 Thus, the coefficient of $\varphi^{123}$ vanishes when $\alpha C=0$. Since $\gamma=0$ we have that $\alpha\not=0$, so $C=0$.
 Similarly, the coefficient of $\varphi^{12\bar3}$ is zero when $\alpha v=0$, so we get $v=0$. But taking $C=v=0$
 one has that the coefficient of $\varphi^{12\bar2}$ does not vanish, because $\alpha\not=0$ and $s^2>0$.
 Hence, there exists no ($G$-left-invariant) \lcht\ structure on $X$.
\end{proof}

In view of the questions in \cite[page 678]{li-zhang} and \cite[Question 1.7]{streets-tian}, and
of the analogous result in \cite[Theorem 1.3]{enrietti-fino-vezzoni} for the non-conformal case, the following question is hence natural.
By Theorem \ref{thm:6-nilmfd-lcht--lck}, it has a positive answer for $6$-dimensional nilmanifolds with left-invariant complex structures.
Also, the differentiable obstruction in Theorem \ref{thm:nilmfd-2step-lcht--heisenberg} suggests an evidence for the question
in case of $2$-step nilmanifolds.

\begin{question}\label{conj:nilmfd-lcht--lck}
For which compact complex manifolds, the existence of locally conformal holomorphic-tamed structures is equivalent to the existence of locally conformal K\"ahler structures?
\end{question}

\begin{rem}
 Consider the deformations in case {\itshape (1)} of the holomorphically parallelizable Nakamura manifold as investigated in \cite{angella-kasuya-2}, see Example \ref{ex:def-1-nakamura}. (Recall that they satisfy the $\del\delbar$-Lemma for $t\neq0$.) We claim that they do not admit any \lcht\ structure. In particular, they do not admit any holomorphic-tamed structure. Indeed, by \cite[\S7]{kasuya-local}, one has $H^1_{dR}(X;\R) = \R \left\langle \de z_1,\, \de \bar z_1 \right\rangle$, and hence the natural map $H^1(\mathfrak{g}) \to H^1_{dR}(X;\R)$ is an isomorphism. Hence, by Lemma \ref{lemma:nilmfd-lcht--inv-lcht}, it suffices to prove that these deformations do not admit any left-invariant \lcht\ structure. Consider the $G$-left-invariant real $2$-form
 \begin{eqnarray*}
  \Omega_t &:=& \im\, \left( A\, \phi^1_t \wedge \bar\phi^1_t + B\, \phi^2_t \wedge \bar\phi^2_t + C\, \phi^3_t \wedge \bar\phi^3_t \right) \\[5pt]
  && + \left( D\, \phi^1_t \wedge \bar\phi^2_t - \bar D\, \phi^2_t \wedge \bar\phi^1_t \right)  + \left( E\, \phi^1_t \wedge \bar\phi^3_t - \bar E\, \phi^3_t \wedge \bar\phi^1_t \right) + \left( F\, \phi^2_t \wedge \bar\phi^3_t - \bar F\, \phi^3_t \wedge \bar\phi^2_t \right) \\[5pt]
  && + \left( L\, \phi^1_t \wedge \phi^2_t + \bar L\, \bar\phi^1_t \wedge \bar\phi^2_t \right) + \left( M\, \phi^1_t \wedge \phi^3_t + \bar M\, \bar\phi^1_t \wedge \bar\phi^3_t \right) + \left( N\, \phi^2_t \wedge \phi^3_t + \bar N\, \bar\phi^2_t \wedge \bar\phi^3_t \right) \;,
 \end{eqnarray*}
 with $A,B,C\in\R$ and $D,E,F\in\C$. The form $\Omega_t$ tames the complex structure $J_t$ (i.e., its $(1,1)$-component with respect to $J_t$ is a positive $(1,1)$-form) if and only if, \cite[page 189]{ugarte},
 $$ \left\{\begin{array}{l}
            A > 0 \\[5pt]
            B > 0 \\[5pt]
            C > 0 \\[5pt]
            AB > |D|^2 \\[5pt]
            AC > |E|^2 \\[5pt]
            BC > |F|^2 \\[5pt]
            ABC + 2\, \Re(\im \bar D E \bar F) > C\,|D|^2 + A\,|F|^2 + B\,|E|^2
    \end{array}\right. \;.
 $$
 A straightforward computation gives
 \begin{eqnarray*}
  \de \Omega_t &=& \left( \bar D - t\, L \right)\, \phi_t^{12\bar1} - \im\, B\, \left( 1 + \bar t \right)\, \phi_t^{12\bar2} - F \left( 1 - \bar t \right)\, \phi_t^{12\bar3} \\[5pt]
  && + \left( - \bar E + t\, M \right)\, \phi_t^{13\bar1} - \bar F\, \left( 1 - \bar t \right)\, \phi_t^{13\bar2} + \im\, C\, \left( 1 + \bar t \right)\, \phi_t^{13\bar3} \\[5pt]
  && + \left( D - \bar t\, \bar L \right)\, \phi_t^{1\bar1\bar2} + \im\, B\, \left( 1 + t \right)\, \phi_t^{2\bar1\bar2} - \bar F\, \left( 1 - t \right)\, \phi_t^{3\bar1\bar2} \\[5pt]
  && + \left( - E + \bar t\, \bar M \right)\, \phi_t^{1\bar1\bar3} - F\, \left( 1 - t \right)\, \phi_t^{2\bar1\bar3} - \im\, C\, \left( 1 + t \right)\, \phi_t^{} \;,
 \end{eqnarray*}
 where we have shortened, e.g., $\phi^{12\bar1\bar3}_t \;:=\; \phi^{1}_{t}\wedge\phi^{2}_{t}\wedge\bar\phi^{1}_{t}\wedge\bar\phi^{3}_{t}$.
 In particular, since $B>0$, then $\Omega_t$ cannot be $\de$-closed for $t$ small.
\end{rem}

\subsection{Locally conformal holomorphic-tamed structures on \texorpdfstring{$4$}{4}-dimensional solvmanifolds}

In \cite[Theorem 1]{hasegawa-jsg}, K. Hasegawa characterized the compact complex surfaces being diffeomorphic to $4$-dimensional solvmanifolds. More precisely, such complex structures turn out to be left-invariant, and six different cases may occur: complex torus, hyperelliptic surface, Inoue surface of type $\mathcal{S}_{M}$, primary Kodaira surface, secondary Kodaira surface, Inoue surface of type $\mathcal{S}^{\pm}$ (see Table \ref{table:eq-struttura-4-solvmfds}). In \cite{angella-dloussky-tomassini}, some results concerning their cohomologies are studied.

Now, we explitly study the existence of \lcht\ structures for such $4$-dimensional solvmanifolds.
For a more general result, see \cite[Theorem 1.1]{apostolov-dloussky} by V. Apostolov and G. Dloussky, proving that any compact complex surface with odd first Betti number admits a \lcht\ structure.
(See, e.g., \cite{belgun, brunella} for further results on \lcK\ metrics for compact complex surfaces.)
The following result provides examples of compact complex surfaces yielding a positive answer to Question \ref{conj:nilmfd-lcht--lck}. This is in accord with \cite[Theorem 7]{belgun}. In view of the recent results by V. Apostolov and G. Dloussky, this result follows from \cite[Theorem 1.1]{apostolov-dloussky}.

\begin{thm}[{see also \cite[Theorem 7]{belgun} and \cite[Theorem 1.1]{apostolov-dloussky}}]\label{thm:4solvmfds-lcht}
 Let $X$ be a compact complex surface diffeomorphic to a solvmanifold.
 Then $X$ admits \lcht\ structures.
 Except in the case of Inoue surface of type $\mathcal{S}^\pm$ with $q\neq0$, then $X$ admits also \lcK\ structures.
\end{thm}

\begin{table}[ht]
 \centering
\begin{tabular}{c || >{$}l<{$} || >{$}l<{$} | >{$}l<{$} ||}
\toprule
 {\bfseries class} & \text{\bfseries conditions} & \de\varphi^1 & \de\varphi^2 \\
\toprule\midrule[0.02em]
{\bfseries complex torus} & & 0 & 0 \\
\midrule[0.02em]
{\bfseries hyperelliptic surface} & & -\frac{1}{2}\,\varphi^{12}+\frac{1}{2}\,\varphi^{1\bar2} & 0 \\
\midrule[0.02em]
{\bfseries Inoue surface $\mathcal{S}_{M}$} & \alpha\in\R\setminus\{0\},\beta\in\R & \frac{\alpha-\im\beta}{2\im}\,\varphi^{12}-\frac{\alpha-\im\beta}{2\im}\,\varphi^{1\bar2} & -\im\,\alpha\,\varphi^{2\bar2} \\
\midrule[0.02em]
{\bfseries primary Kodaira surface} & & 0 & \frac{\im}{2}\,\varphi^{1\bar1} \\
\midrule[0.02em]
{\bfseries secondary Kodaira surface} & & -\frac{1}{2}\,\varphi^{12}+\frac{1}{2}\,\varphi^{1\bar2} & \frac{\im}{2}\,\varphi^{1\bar1} \\
\midrule[0.02em]
{\bfseries Inoue surface $\mathcal{S}^{\pm}$} & q\in\R & \frac{1}{2\im}\,\varphi^{12}+\frac{1}{2\im}\,\varphi^{2\bar1}+\frac{q\im}{2}\,\varphi^{2\bar2} & \frac{1}{2\im}\,\varphi^{2\bar2} \\
\midrule[0.02em]
\bottomrule
\end{tabular}
\caption{Structure equations for the compact complex surfaces being diffeomorphic to $4$-dimensional solvmanifolds, with respect to a left-invariant co-frame $\left\{\varphi^1,\varphi^2\right\}$ of $(1,0)$-forms, as classified by K. Hasegawa in \cite{hasegawa-jsg}.}
\label{table:eq-struttura-4-solvmfds}
\end{table}

\begin{proof}
 Consider a left-invariant co-frame $\left\{\varphi^1, \varphi^2\right\}$ of $(1,0)$-forms. We recall in Table \ref{table:eq-struttura-4-solvmfds} the structure equations with respect to such co-frame, as in \cite[Theorem 1]{hasegawa-jsg}. Denote by $\g$ the associated Lie algebra. Since the natural map $H^1(\mathfrak{g}) \to H^1_{dR}(X;\R)$ is always an isomorphism, see, e.g., \cite[Theorem 4.1]{angella-dloussky-tomassini}, we can use Lemma \ref{lemma:nilmfd-lcht--inv-lcht} in order to reduce the computations to left-invariant \lcht\ structures. More precisely, consider the left-invariant real $2$-form
 $$ \Omega \;=\; \im\, A\, \varphi^{1\bar1} + \im\, B\, \varphi^{2\bar2} + \left( D\, \varphi^{1\bar2} - \bar D\, \varphi^{2\bar1} \right) + \left( L\, \varphi^{12} + \bar L\, \varphi^{\bar1\bar2} \right) \;, $$
 where $A,B\in\R$ and $D,L\in\C$. (As a matter of notation, we shorten, e.g., $\varphi^{1\bar2} := \varphi^1 \wedge \bar\varphi^2$.) The condition that $\Omega$ tames the complex structure is shown to be equivalent to
 $$ \left\{\begin{array}{l}
            A > 0 \\[5pt]
            B > 0 \\[5pt]
            AB > |D|^2
    \end{array}\right. \;.
 $$
 Consider the left-invariant real $1$-form
 $$ \vartheta \;=\; a\, \varphi^1 + b\, \varphi^2 + \bar a\, \bar\varphi^1 + \bar b\, \bar\varphi^2 \;. $$
 We compute
 \begin{eqnarray*}
  \vartheta \wedge \Omega &=& \left( - \im\, A\, b - \bar D\, a + L\, \bar a \right)\, \varphi^{12\bar1} + \left( \im\, B\, a - D\, b + L\, \bar b \right)\, \varphi^{12\bar2} \\[5pt]
  && + \left( \im\, A\, \bar b - D\, \bar a + \bar L\, a \right)\, \varphi^{1\bar1\bar2} + \left( - \im\, B\, \bar a - \bar D\, \bar b + \bar L\, b \right)\, \varphi^{2\bar1\bar2} \;.
 \end{eqnarray*}
 We study each case in the classification by \cite[Theorem 1]{hasegawa-jsg} separately.

\caso{Torus} The torus, which has structure equations
 $$ \de\varphi^1 \;=\; \de\varphi^2 \;=\; 0 \;, $$
 clearly admits a K\"ahler structure.

\caso{Hyperelliptic surface} The hyperelliptic surface is characterized by the structure equations
 $$ \de\varphi^1 \;=\; -\frac{1}{2}\, \varphi^{12} + \frac{1}{2}\, \varphi^{1\bar2} \;, \qquad \de\varphi^2 \;=\; 0 \;. $$
 It clearly admits a K\"ahler structure. More precisely, an explicit computation gives
 $$ \de \Omega \;=\; - \frac{1}{2}\, \left( D + L \right) \, \varphi^{12\bar2} - \frac{1}{2}\, \left( \bar D + \bar L \right)\, \varphi^{2\bar1\bar2} \;. $$
 We compute
 $$ \de\vartheta \;=\; -\frac{a}{2}\, \varphi^{12} + \frac{a}{2}\, \varphi^{1\bar2} - \frac{\bar a}{2}\, \varphi^{2\bar1} - \frac{\bar a}{2}\, \varphi^{\bar1\bar2} \;. $$
 In particular, the condition $\de\vartheta=0$ is equivalent to
 $$ a \;=\; 0 \;. $$
 The condition $\de_\vartheta\Omega=0$ with $\de\vartheta=0$ is equivalent to
 $$ \left\{ \begin{array}{l}
     - \im\, A\, b \;=\; 0 \\[5pt]
     - D\, b + L\, \bar b \;=\; - \frac{1}{2}\, \left( D + L \right)
    \end{array}\right. \;.
 $$
 In particular, it admits both \gcK\ structures, with
 $$ A > 0 , \quad B > 0 , \quad D = L = 0 \qquad \text{ and } \qquad a = b = 0 \;, $$
 and \gcht\ structures, with
 $$ A > 0 , \quad B > 0 , \quad D = - L , \quad AB > |D|^2 \qquad \text{ and } \qquad a = b = 0 \;. $$

\caso{Inoue surface $\mathcal{S}_{M}$} The Inoue surfaces of type $\mathcal{S}_{M}$ are characterized by the structure equations
 $$ \de\varphi^1 \;=\; \frac{\alpha-\im\beta}{2\im}\, \varphi^{12} - \frac{\alpha-\im\beta}{2\im}\, \varphi^{1\bar2} \;, \qquad \de\varphi^2 \;=\; -\im\alpha\, \varphi^{2\bar2} \;, $$
 where $\alpha\in\R\setminus\{0\}$ and $\beta\in\R$.
 An explicit computation gives
  $$ \de \Omega \;=\; \left( \alpha\, A \right)\, \varphi^{12\bar1} + \frac{D+L}{2}\, \left( -\beta + \im \alpha \right)\, \varphi^{12\bar2}  + \left( \alpha\, A  \right)\, \varphi^{1\bar1\bar2} - \frac{\bar D + \bar L}{2}\, \left( \beta + \im \alpha \right)\, \varphi^{2\bar1\bar2} \;. $$
 We compute
 $$ \de\vartheta \;=\; \frac{\alpha - \im \beta}{2\im}\, a\, \varphi^{12} - \frac{\alpha - \im \beta}{2\im}\, a\, \varphi^{1\bar2} - \frac{\alpha + \im \beta}{2\im}\, \bar a\, \varphi^{2\bar1} - 2\, \im\alpha\, \Re b\, \varphi^{2\bar2} - \frac{\alpha + \im \beta}{2\im}\, \bar a\, \varphi^{\bar1\bar2} \;. $$
 In particular, since $\alpha\in\R\setminus\{0\}$, the condition $\de\vartheta=0$ is equivalent to
 $$ a \;=\; \Re b \;=\; 0 \;. $$
 The condition $\de_\vartheta\Omega=0$ with $\de\vartheta=0$ is equivalent to
 $$ \left\{ \begin{array}{l}
     - \im\, A\, b \;=\; \alpha\, A \\[5pt]
     - \left( D + L \right)\, b \;=\; \frac{D+L}{2}\, \left( -\beta + \im\alpha \right)
    \end{array}\right. \;.
 $$
 Hence it admits both \lcK\ structures, with
 $$ A > 0 , \quad B > 0 , \quad D = 0 , \quad L = 0 , \qquad \text{ and } \qquad a = 0 , \quad b = \im\, \alpha \;, $$
 and \lcht\ structures, with
 $$ A > 0 , \quad B > 0 , \quad AB > |D|^2 , \quad L = - D , \qquad \text{ and } \qquad a = 0 , \quad b = \im\, \alpha \;. $$

\caso{Primary Kodaira surface} The primary Kodaira surface is characterized by the structure equations
 $$ \de\varphi^1 \;=\; 0 \;, \qquad \de\varphi^2 \;=\; \frac{\im}{2}\, \varphi^{1\bar1} \;. $$
 An explicit computation gives
 $$ \de \Omega \;=\; - \frac{B}{2}\, \varphi^{12\bar1} - \frac{B}{2}\, \varphi^{1\bar1\bar2} \;. $$
 We compute
 $$ \de\vartheta \;=\; \im\, \Re b\, \varphi^{1\bar1} \;. $$
 In particular, the condition $\de\vartheta=0$ is equivalent to
 $$ \Re b \;=\; 0 \;. $$
 The condition $\de_\vartheta\Omega=0$ with $\de\vartheta=0$ is equivalent to
 $$ \left\{ \begin{array}{l}
     - \im\, A\, b - \bar D\, a + L\, \bar a \;=\; - \frac{B}{2} \\[5pt]
     \im\, B\, a - \left( D + L \right)\, b \;=\; 0
    \end{array}\right. \;.
 $$
 In particular, it admits both \lcK\ structures, with
 $$ A > 0 , \quad B > 0 , \quad AB > |D|^2 , \quad L = 0 , \qquad \text{ and } \qquad a = - \frac{BD}{2\left(AB-|D|^2\right)} , \quad b = - \frac{\im\,B^2}{2\left(AB-|D|^2\right)} \;, $$
 and \lcht\ structures, with
 $$ A > 0 , \quad B > 0 , \quad AB > |D|^2 , \quad L \in \C , \qquad \text{ and } \qquad a = - \frac{B\left(D+L\right)}{2\left(AB-|D|^2+|L|^2\right)} , \quad b = - \frac{\im\,B^2}{2\left(AB-|D|^2+|L|^2\right)} \;. $$

\caso{Secondary Kodaira surface} The secondary Kodaira surface is characterized by the structure equations
 $$ \de\varphi^1 \;=\; -\frac{1}{2}\, \varphi^{12} + \frac{1}{2}\, \varphi^{1\bar2} \;, \qquad \de\varphi^2 \;=\; \frac{\im}{2}\, \varphi^{1\bar1} \;. $$
 An explicit computation gives
 $$ \de \Omega \;=\; - \frac{B}{2}\, \varphi^{12\bar1} - \frac{D + L}{2}\, \varphi^{12\bar2} - \frac{B}{2}\, \varphi^{1\bar1\bar2} - \frac{\bar D + \bar L}{2}\, \varphi^{2\bar1\bar2} \;. $$
 We compute
 $$ \de\vartheta \;=\; - \frac{a}{2}\, \varphi^{12} + \im\, \Re b\, \varphi^{1\bar1} + \frac{a}{2}\, \varphi^{1\bar2} - \frac{\bar a}{2}\, \varphi^{2\bar1} - \frac{\bar a}{2}\, \varphi^{\bar1\bar2} \;. $$
 In particular, the condition $\de\vartheta=0$ is equivalent to
 $$ a \;=\; \Re b \;=\; 0 \;. $$
 The condition $\de_\vartheta\Omega=0$ with $\de\vartheta=0$ is equivalent to
 $$ \left\{ \begin{array}{l}
     - \im\, A\, b \;=\; - \frac{B}{2} \\[5pt]
     - \left( D + L \right) \, b \;=\; - \frac{D + L}{2}
    \end{array}\right. \;.
 $$
 In particular, it admits both \lcK\ structures, with
 $$ A > 0 , \quad B > 0 , \quad D = L = 0 , \qquad \text{ and } \qquad a = 0 , \quad b = - \frac{B\, \im}{2\, A} \;, $$
 and \lcht\ structures, with
 $$ A > 0 , \quad B > 0 , \quad \quad AB > |D|^2 , \quad L = - D , \qquad \text{ and } \qquad a = 0 , \quad b = - \frac{B\, \im}{2\, A} \;. $$

\caso{Inoue surface $\mathcal{S}^{\pm}$} The Inoue surface of type $\mathcal{S}^{\pm}$ is characterized by the structure equations
 $$ \de\varphi^1 \;=\; \frac{1}{2\im}\, \varphi^{12}+\frac{1}{2\im}\, \varphi^{2\bar1}+\frac{q\im}{2}\,\varphi^{2\bar2} \;, \qquad \de\varphi^2 \;=\; \frac{1}{2\im}\, \varphi^{2\bar2} \;, $$
 where $q\in\R$.
 An explicit computation gives
 $$ \de \Omega \;=\; \frac{A}{2}\, \varphi^{12\bar1} + \frac{1}{2}\, \left( \im\, \left( L + \bar D \right) + A\, q \right)\, \varphi^{12\bar2} + \frac{A}{2}\, \varphi^{1\bar1\bar2} + \frac{1}{2}\, \left( - \im\, \left( \bar L + D \right) + A\, q \right)\, \varphi^{2\bar1\bar2} \;. $$
 We compute
 $$ \de\vartheta \;=\; \frac{a}{2\im}\, \varphi^{12} + \frac{\bar a}{2\im}\, \varphi^{1\bar2} + \frac{a}{2\im}\, \varphi^{2\bar1} + \im\, \left( q\,\Re a - \Re b\right)\, \varphi^{2\bar2} - \frac{\bar a}{2\im}\, \varphi^{\bar1\bar2} \;. $$
 In particular, the condition $\de\vartheta=0$ is equivalent to
 $$ a \;=\; \Re b \;=\; 0 \;. $$
 The condition $\de_\vartheta\Omega=0$ with $\de\vartheta=0$ is equivalent to
 $$ \left\{ \begin{array}{l}
     - \im\, A\, b \;=\; \frac{A}{2} \\[5pt]
     - \left( D + L \right)\, b \;=\; \frac{1}{2}\, \left( \im\, \left( L + \bar D \right) + A\, q \right)
    \end{array}\right. \;.
 $$
 In particular:
 \begin{itemize}
  \item in the case $q=0$:
  it admits both \lcK\ structures, with
  $$ A > 0 , \quad B > 0 , \quad AB > |D|^2 , \quad \Re D = 0 , \quad L = 0 , \qquad \text{ and } \qquad a = 0 , \quad b = \frac{\im}{2} \;, $$
  and \lcht\ structures, with
  $$ A > 0 , \quad B > 0 , \quad AB > |D|^2 , \quad L = - \Re D, \qquad \text{ and } \qquad a = 0 , \quad b = \frac{\im}{2} \;; $$
  \item in the case $q\neq0$:
  it does not admit any \lcK\ structure, (the equation $Aq=-2\im\Re D$ yielding $A=0$ that is not admissible,)
  but it admits \lcht\ structures, with
  $$ A > 0 , \quad B > 0 , \quad AB > |D|^2 , \quad L = - \Re D + \frac{\im}{2} A q , \qquad \text{ and } \qquad a = 0 , \quad b = \frac{\im}{2} \;. $$
 \end{itemize}

This concludes the proof.
\end{proof}

\subsection{Locally conformal holomorphic-tamed structures on
\texorpdfstring{$6$}{6}-dimensional solvmanifolds with invariant
complex structures with holomorphically trivial canonical bundle}

In the direction of Question~\ref{conj:nilmfd-lcht--lck}, we investigate the class of
$6$-dimensional solvmanifolds obtained in \cite{fino-otal-ugarte}, see also \cite{otal-phd}.
More precisely, the $6$-dimensional unimodular solvable Lie algebras admitting a linear complex structure and a non-vanishing $\de$-closed $(3,0)$-form are classified up to isomorphisms in \cite[Theorem 2.8]{fino-otal-ugarte}, and the moduli of left-invariant complex structures on the corresponding solvmanifolds are classified in \cite[Theorem 3.10]{fino-otal-ugarte}. In a sense, such complex structures are a very first generalization of left-invariant complex structures on nilmanifolds. Furthermore, they provide interesting examples of solvmanifolds satisfying the $\del\delbar$-Lemma.

These complex structures on Lie algebras are divided into seven classes. By considering a basis $\left\{\omega^{1}, \omega^{2}, \omega^3\right\}$ of the $(1,0)$-forms, we recall in Table \ref{table:otal-solvmfds} the associated structure equations. (As a matter of notation, here and in the following, we shorten, e.g., $\omega^{1\bar2}:=\omega^1\wedge\bar\omega^2$.) Compare also \cite[Proposition 3.3, Proposition 3.4, Proposition 3.6, Proposition 3.7, Proposition 3.9, Theorem 3.10]{fino-otal-ugarte}.

\begin{table}[ht]
 \centering
\resizebox{\textwidth}{!}{
\begin{tabular}{c || >{$}l<{$} || >{$}l<{$} | >{$}l<{$} | >{$}l<{$} ||}
\toprule
 {\bfseries class} & \text{\bfseries conditions} & \de\omega^1 & \de\omega^2 & \de\omega^3 \\
\toprule\midrule[0.02em]
\multirow{2}{*}{\bfseries class (1)} & A=\cos\theta+\im\sin\theta, & \multirow{2}{*}{$A\,\omega^{13}+A\,\omega^{1\bar3}$} & \multirow{2}{*}{$-A\,\omega^{23}-A\,\omega^{2\bar3}$} & \multirow{2}{*}{$0$} \\
& \theta\in[0,\pi) & & & \\
\midrule[0.02em]
{\bfseries class (2)} & g>0 & 0 & -\frac{1}{2}\,\omega^{13}-\left(\frac{1}{2}+g\,\im\right)\,\omega^{1\bar3}+g\,\im\,\omega^{3\bar1} & \frac{1}{2}\,\omega^{12}+\left(\frac{1}{2}-\frac{\im}{4\,g}\right)\,\omega^{1\bar2}+\frac{\im}{4\,g}\,\omega^{2\bar1} \\
\midrule[0.02em]
\multirow{3}{*}{\bfseries class (3)} & A\in\C,\sigma_{12}\in\C, \sigma_{11}\in\R,\sigma_{22}\in\R, & \multirow{2}{*}{$A\,\omega^{13}+A\,\omega^{1\bar3}$} & \multirow{2}{*}{$-A\,\omega^{23}-A\,\omega^{2\bar3}$} & \multirow{2}{*}{$\sigma_{11}\,\omega^{1\bar1}+\sigma_{12}\,\omega^{1\bar2}+\bar\sigma_{12}\,\omega^{2\bar1}+\sigma_{22}\,\omega^{2\bar2}$} \\
& \Re A \sigma_{11}=0, \Re A \sigma_{22}=0, \Im A \sigma_{12}=0, & & & \\
& |A|=1, \left( \sigma_{11}, \sigma_{22}, \sigma_{12} \right) \neq (0,0,0) & & & \\
\midrule[0.02em]
{\bfseries class (4)} & \Im A \neq 0 & -(A-\im)\,\omega^{13} - (A+\im)\, \omega^{1\bar3} & (A-\im)\, \omega^{23}+(A+\im)\,\omega^{2\bar3} & 0 \\
\midrule[0.02em]
{\bfseries class (5)} & \varepsilon = 0 & \multirow{2}{*}{$2\,\im\,\omega^{13}+\omega^{3\bar3}$} & \multirow{2}{*}{$-2\im\,\omega^{23}+\varepsilon\,\omega^{3\bar3}$} & \multirow{2}{*}{$0$} \\
\text{\bfseries class (6)} & \varepsilon = 1 & & & \\
\midrule[0.02em]
{\bfseries class (7)} & & -\omega^{3\bar3} & -\frac{\im}{2}\,\omega^{2\bar1}+\frac{1}{2}\,\omega^{1\bar3}+\frac{\im}{2}\,\omega^{12} & \frac{\im}{2}\,\omega^{3\bar1}-\frac{\im}{2}\,\omega^{13} \\
\midrule[0.02em]
\bottomrule
\end{tabular}
}
\caption{Structure equations for the seven classes of linear complex structures on $6$-dimensional solvable Lie algebras admitting a non-vanishing $\de$-closed $(3,0)$-form, with respect to a co-frame $\left\{\omega^1, \omega^2, \omega^3\right\}$ of $(1,0)$-forms,
as classified in \cite{fino-otal-ugarte}, see also \cite{otal-phd}.}
\label{table:otal-solvmfds}
\end{table}

\begin{table}[ht]
 \centering
\resizebox{\textwidth}{!}{
\begin{tabular}{>{$\mathbf\bgroup}l<{\mathbf\egroup$} || >{$}l<{$} || >{$}l<{$} | >{$}l<{$} | >{$}l<{$} | >{$}l<{$} | >{$}l<{$} | >{$}l<{$} ||}
\toprule
\omega^{jhk} & \vartheta\wedge\Omega & \multicolumn{6}{c||}{$\de\Omega$} \\
 & & \text{class (1)} & \text{class (2)} & \text{class (3)} & \text{class (4)} & \text{classes (5) and (6)} & \text{class (7)} \\
\toprule\midrule[0.02em]
\text{conditions} & & A=\cos\theta+\im\sin\theta, & g>0 & A\in\C,\sigma_{12}\in\C, \sigma_{11}\in\R,\sigma_{22}\in\R, & \Im A\neq0 & \varepsilon \in \{0,1\} & \\
 & & \theta\in[0,\pi) & & \Re A \sigma_{11}=0, \Re A \sigma_{22}=0, \Im A \sigma_{12}=0, & & & \\
 & & & & |A|=1, \left( \sigma_{11}, \sigma_{22}, \sigma_{12} \right) \neq (0,0,0) & & & \\
\hline
\de\theta=0 & & a=0, b=0 & b=0, c=0 & a=0, b=0, c=\bar c & a=0, b=0 & a=0, b=0 & a=\bar a, b=0, c=0 \\
\toprule
\midrule[0.02em]
123 & cL+aN-bM & 0 & 0 & 0 & 0 & 0 & 0 \\
\midrule[0.02em]\midrule[0.02em]
12\bar1 & -a\bar u+\bar aL-\im br^2 & 0 & -\frac{\im}{4g}M+\frac{2g+\im}{4g}z-\frac{1}{2}\bar z & -\bar\sigma_{12}M+\sigma_{11}N+\bar\sigma_{12}z-\sigma_{11}v & 0 & 0 & \frac{\im}{2}L+\frac{\im}{2}\bar u \\
12\bar2 & \im as^2+\bar b L-bu & 0 & \frac{2g-\im}{4g}N+\frac{\im}{4g}v-\frac{1}{2}\bar v & -\sigma_{22}M+\sigma_{12}N+\sigma_{22}z-\sigma_{12}v & 0 & 0 & 0 \\
12\bar3 & \bar cL+av-bz & 0 & -gs^2+\frac{\im}{2}t^2 & 0 & 0 & 0 & \im v \\
\midrule[0.02em]
13\bar1 & \bar aM-a\bar z-\im cr^2 & 2\Re A\im r^2 & -\im gL-\left(\frac{1}{2}-\im g\right)u+\frac{1}{2}\bar u & 2\Re A\im r^2-\sigma_{11} \im t^2 & -2(\Re A - \im)\im r^2 & -2 r^2 & -\frac{\im}{2}M+\frac{1}{2}u+\frac{\im}{2}\bar z \\
13\bar2 & -a\bar v+\bar bM-cu & 2\Im A \im u & -\frac{\im}{2}s^2-\frac{1}{4g}t^2 & 2\Im A \im u-\sigma_{12}\im t^2 & -2\Im A\im u & 2\im u & \im\bar v \\
13\bar3 & \bar cM+\im at^2-cz & -MA+zA & \left(\frac{1}{2}+\im g\right)N-\frac{1}{2}v-\im g\bar v & -MA+zA & (A-\im)M-(A-\im)z & -\varepsilon L+\im r^2+\varepsilon u+2\im z & -\frac{1}{2}N-\im r^2 \\
\midrule[0.02em]
23\bar1 & \bar aN-b\bar z +c\bar u & 2\Im A \im \bar u & -\left(\frac{1}{2}-\im g\right)\im s^2-\frac{2g+\im}{4g}\im t^2 & 2\Im A \im \bar u-\bar\sigma_{12}\im t^2 & -2\Im A \im\bar u & 2\im\bar u & \frac{\im}{2}s^2 \\
23\bar2 & \bar b N - b\bar v -\im cs^2 & -2\Re A\im s^2 & 0 & -2\Re A\im s^2-\sigma_{22}\im t^2 & 2(\Re A-\im)\im s^2 & 2s^2 & 0 \\
23\bar3 & \bar c N + \im bt^2-cv & NA-vA & 0 & NA-vA & -(A+\im)N+(A-\im)v & L-\bar u +\varepsilon\im s^2 & -L+\bar u \\
\midrule[0.02em]\midrule[0.02em]
1\bar1\bar2 & \im\bar b r^2+a\bar L-\bar a u & 0 & -\frac{1}{2}z+\frac{2g-\im}{4g}\bar z+\frac{\im}{4g}\bar u & \sigma_{12}\bar z-\sigma_{11}\bar v-\sigma_{12}\bar M+\sigma_{11}\bar N & 0 & 0 & \frac{\im}{2}u-\frac{\im}{2}\bar L \\
1\bar1\bar3 & a\bar M-\bar a z+\im\bar c r^2 & -2\Re A\im r^2 & \frac{1}{2}u-\left(\frac{1}{2}+\im g\right)\bar u+\im g\bar L & -2\Re A\im r^2+\sigma_{11}\im t^2 & 2(\Re A+\im)\im r^2 & -2r^2 & -\frac{\im}{2}z+\frac{1}{2}\bar u+\frac{\im}{2}\bar M \\
1\bar2\bar3 & \bar b z -\bar c u - a\bar N & -2\Im A\im u & \left(\frac{1}{2}+\im g\right)\im s^2+\frac{2g-\im}{4g}\im t^2 & -2\Im A\im u +\sigma_{12}\im t^2 & 2\Im A\im u & -2\im u & -\frac{\im}{2}s^2 \\
\midrule[0.02em]
2\bar1\bar2 & -\bar b \bar u + b\bar L -\im\bar a s^2 & 0 & -\frac{1}{2}v-\frac{\im}{4g}\bar v+\frac{2g+\im}{4g}\bar N & \sigma_{22}\bar z-\bar\sigma_{12}\bar v-\sigma_{22}\bar M+\bar\sigma_{12}\bar N & 2\Im A\im \bar u & 0 & 0 \\
2\bar1\bar3 & -\bar c\bar u +b\bar M-\bar a v & -2\Im A\im \bar u & \frac{\im}{2}s^2-\frac{1}{4g}t^2 & -2\Im A\im \bar u+\bar\sigma_{12}\im t^2 & 0 & -2\im \bar u & -\im v \\
2\bar2\bar3 & \im \bar cs^2-\bar bv+b\bar N & 2\Re A\im s^2 & 0 & 2\Re A\im s^2+\sigma_{22}\im t^2 & -2(\Re A+\im)\im s^2 & 2s^2-2\im v & 0 \\
\midrule[0.02em]
3\bar1\bar2 & -\bar b \bar z+\bar a \bar v +c\bar L & 0 & -gs^2-\frac{\im}{2}t^2 & 0 & 0 & 0 & -\im \bar v \\
3\bar1\bar3 & c\bar M -\bar c\bar z -\im\bar a t^2 & \bar A \bar z -\bar A\bar M & \im g v-\frac{1}{2}\bar v+\left(\frac{1}{2}-\im g\right)\bar N & \bar A \bar z - \bar A\bar M & -(\bar A+\im)\bar z + (\bar A -\im) \bar M & -\im r^2+\varepsilon \bar u-2\im\bar z-\varepsilon\bar L & \im r^2-\frac{1}{2}\bar N \\
3\bar2\bar3 & \im \bar b t^2 +\bar c \bar v -c\bar N & -\bar A\bar v+\bar A\bar N & 0 & -\bar A\bar v+\bar A\bar N & (\bar A+\im)\bar v -(\bar A-\im) \bar N & -u-\varepsilon\im s^2+2\im\bar v+\bar M & u-\bar L \\
\midrule[0.02em]\midrule[0.02em]
\bar1\bar2\bar3 & \bar b \bar M -\bar c \bar L -\bar a \bar N & 0 & 0 & 0 & 0 & 0 & 0 \\
\midrule[0.02em]
\bottomrule
\end{tabular}
}
\caption{Components of $\vartheta\wedge\Omega$ and $\de\Omega$ for the linear complex structures on $6$-dimensional solvable Lie algebras admitting a non-vanishing $\de$-closed $(3,0)$-form as studied in \cite{fino-otal-ugarte}, see also \cite{otal-phd}. (See proof of Theorem \ref{thm:solmfd-otal-lcht-lck} for notations.)}
\label{table:otal-solvmfds-dOmega}
\end{table}

In \cite{fino-otal-ugarte,otal-phd}, some results on Hermitian metrics and cohomological properties of the above solvmanifolds with left-invariant complex structures are studied. We prove here the following result, on the existence of \lcht\ and \lcK\ structures.

\begin{thm}\label{thm:solmfd-otal-lcht-lck}
 Consider the left-invariant complex structures on $6$-dimensional solvmanifolds with holomorphically trivial canonical bundle
 as classified in \cite{fino-otal-ugarte}, see also \cite{otal-phd}. According to that classification, they are divided into seven classes, see Table \ref{table:otal-solvmfds}.
 \begin{itemize}
  \item The complex structures in class (1) admit a linear \lcht\ structure if and only if they admit a linear \lcK\ structure if and only if $A=\im$.
  \item The complex structures in class (3) admit a linear \lcht\ structure if and only if they admit a linear \lcK\ structure if and only if $A\in\{\im,-\im\}$.
  \item The complex structures in classes (2), (4), (5), (6), (7) do not admit either any linear \lcht\ structure or any linear \lcK\ structure.
 \end{itemize}
\end{thm}

\begin{proof}
 As a matter of notation, consider the real $2$-form
 \begin{eqnarray*}
  \Omega &=& \left( \im\, r^2\, \omega^{1\bar1} + \im\, s^2\, \omega^{2\bar2} + \im\, t^2\, \omega^{3\bar3} \right) \\[5pt]
         && + \left( u\, \omega^{1\bar2} - \bar u\, \omega^{2\bar1} \right)
            + \left( v\, \omega^{2\bar3} - \bar v\, \omega^{3\bar2} \right)
            + \left( z\, \omega^{1\bar3} - \bar z\, \omega^{3\bar1} \right) \\[5pt]
         && + \left( L\, \omega^{12} + M\, \omega^{13} + N\, \omega^{23} \right)
            + \left( \bar L\, \omega^{\bar1\bar2} + \bar M\, \omega^{\bar1\bar3} + \bar N\, \omega^{\bar2\bar3} \right) \;,
 \end{eqnarray*}
 where $L,M,N\in\C$, and $r,s,t\in\R$ and $u,v,z\in\C$ satisfy, \cite[page 189]{ugarte},
 $$ r > 0 , \quad s > 0 , \quad t > 0 , \quad rs > |u|^2 , \quad st > |v|^2 , \quad rt > |z|^2 , \quad rst+2\Re(\im\bar u\bar v z)>t|u|^2+r|v|^2+s|z|^2 \;. $$
 Consider also
 $$ \vartheta \;=\; a\, \omega^{1} + b\, \omega^{2} + c\, \omega^{3}
                    + \bar a\, \omega^{\bar1} + \bar b\, \omega^{\bar 2} + \bar c\, \omega^{\bar 3} $$
 where $a,b,c \in \C$.

 We summarize the components of $\vartheta\wedge\Omega$ and of $\de\Omega$ for the each of the seven classes in Table \ref{table:otal-solvmfds-dOmega}. We recall that the existence of \lcht\ structures is equivalent to solve the equation $\de\Omega=\vartheta\wedge\Omega$ in $r,s,t,u,v,z,a,b,c$ satisfying the conditions above. We may further assume $\de\vartheta=0$. Hence we consider now these equations for each case separately.

 \caso{Class (1)} By matching the coefficients of $\omega^{13\bar1}$, we get $c=-2\Re A$. By matching the coefficients of $\omega^{23\bar2}$, we get $c=2\Re A$. Hence $\Re A=0$. In fact, $A=\im$. In this case, the coefficients of $\omega^{13\bar2}$, $\omega^{13\bar3}$, and $\omega^{23\bar3}$ give $u=0$, $M=z$, $N=v$, respectively, and the system reduces to these equations. In particular, for the complex structures with $A=\im$ in class (1), there exists both \lcht\ and \lcK\ structures.

 \caso{Class (2)} By matching the coefficients of $\omega^{13\bar2}$, we get $\frac{\im}{2}s^2+\frac{1}{4g}t^2=a\bar v$. Hence $a\neq0$. Therefore, by matching the coefficients of $\omega^{123}$, that is, $aN=0$, it follows that $N=0$. Consider the coefficients of $\omega^{23\bar1}$: we get $-\left(\frac{1}{2}-\im g\right)\im s^2-\frac{2g+\im}{4g}\im t^2=\bar aN=0$. Since $\Im \left(-\left(\frac{1}{2}-\im g\right)\im s^2-\frac{2g+\im}{4g}\im t^2\right)=-\frac{\im}{2}s^2-\frac{\im}{2}t^2\neq 0$, we get an absurd. Therefore there exists no \lcht\ structure for the complex structures in class (2).

 \caso{Class (3)} By the conditions on the parameters, two cases may occur.

 \begin{itemize}
  \item Assume $\Im A\neq 0$. Then $\sigma_{12}=0$. Since $\left( \sigma_{11}, \sigma_{12}, \sigma_{22} \right) \neq (0,0,0)$, then either $\sigma_{11}$ or $\sigma_{22}$ is non-zero. It follows that $\Re A=0$. In fact, $A\in\{\im,-\im\}$. By matching the coefficients of $\omega^{13\bar1}$, we get $c=\sigma_{11}\frac{t^2}{r^2}\in\R$. By matching the coefficients of $\omega^{23\bar2}$, we get $c=\sigma_{22}\frac{t^2}{s^2}\in\R$. By matching the coefficients of $\omega^{13\bar2}$, $\omega^{13\bar3}$, and $\omega^{23\bar3}$, we get, respectively, $u=0$, $z=M$, $v=N$. By matching the coefficients of $\omega^{123}$, we get $L=0$. The system reduces to these equations.
  \item Assume $\Im A=0$. Then $\sigma_{11}=0$ and $\sigma_{22}=0$. By matching the coefficients of $\omega^{13\bar1}$ and $\omega^{23\bar2}$, we get, respectively, $c=-2\Re A$ and $c=2\Re A$. Hence it follows $\Re A=0$. This is absurd, therefore there is no \lcht\ structure in this case.
 \end{itemize}

 \caso{Class (4)} By matching the coefficients of $\omega^{13\bar1}$ and $\omega^{23\bar2}$, we get, respectively, $c=2(\Re A-\im)$ and $c=-2(\Re A-\im)$. Hence it follows $\Re A-\im=0$. This is absurd, therefore there is no \lcht\ structure in this case.

 \caso{Classes (5) and (6)} By matching the coefficients of $\omega^{13\bar1}$ and $\omega^{23\bar2}$, we get, respectively, $c=-2\im$ and $c=2\im$. This is absurd, therefore there is no \lcht\ structure in this case.

 \caso{Class (7)} By matching the coefficients of $\omega^{12\bar2}$, we get $a\im s^2=0$. Hence $a=0$. But then, by matching the coefficients of $\omega^{23\bar1}$, we get $\frac{\im}{2}s^2=0$. This is absurd, therefore there is no \lcht\ structure in this case.

This concludes the proof.
\end{proof}

Summarizing, we have the following result, which provides a further class yielding a positive answer to Question \ref{conj:nilmfd-lcht--lck}, namely, invariant structures on the solvmanifolds studied in \cite{fino-otal-ugarte}, see also \cite{otal-phd}.

\begin{cor}\label{cor:solmfd-otal-lcht-lck-final}
Let $X$ be a $6$-dimensional solvmanifold endowed with a left-invariant
complex structure with holomorphically trivial canonical bundle.
Then, $X$ admits a linear \lcht\ structure if and only if it admits a linear \lcK\ structure.
\end{cor}

\end{document}